\documentclass[12pt, a4paper]{article}
\usepackage{amsfonts,amssymb,amsmath,amscd,latexsym,makeidx,theorem}

\usepackage{hyperref}
\usepackage{color}
\usepackage{txfonts}

\newtheorem{thm}{Theorem}[section]
\newtheorem{thmx}{Theorem}

\newtheorem{lem}[thm]{Lemma}
\newtheorem{prop}[thm]{Proposition}
\newtheorem{cor}[thm]{Corollary}

\newenvironment{proof}{\noindent\emph{Proof.}}{\hfill$\square$\medskip}
\newenvironment{rmk}{\medskip\noindent\emph{Remark.}}{\medskip}

\newcommand{\M}[1]{\mathcal{#1}}

\newcommand{\sph}{\mathbb{S}}
\newcommand{\de}{\partial}
\newcommand{\D}{\Delta}
\newcommand{\vp}{\varphi}
\newcommand{\R}{\mathbb{R}}
\newcommand{\ve}{\varepsilon}

\newcommand{\cl}[1]{\overline{#1}}

\renewcommand{\(}{\left(}
\renewcommand{\)}{\right)}

\DeclareMathOperator{\av}{av}
\DeclareMathOperator{\loc}{loc}

\author{Ali Hyder\thanks{The author is partially supported by SERB SRG/2022/001291} \\ {\small Tata Institute of Fundamental Research}\\ {\small \texttt{hyder@tifrbng.res.in}} \and Luca Martinazzi\thanks{The author is partially supported by Fondazione CARIPLO and Fondazione CDP grant n. 2022-2118, and by the PRIN Project 2022PJ9EFL \emph{Geometric Measure Theory: Structure of Singular Measures, Regularity Theory and Applications in the Calculus of Variations}.}\\  {\small Sapienza, Universit\`a di Roma} \\ {\small \texttt{luca.martinazzi@uniroma1.it}}}

\title{One-dimensional half-harmonic maps into the circle and their degree}

\begin{document}

\maketitle

\begin{abstract}
Given a half-harmonic map $u\in \dot H^{\frac{1}{2},2}(\R,\sph^1)$ minimizing the fractional Dirichlet energy under Dirichlet boundary conditions in $\R\setminus I$, we show the existence of a second half-harmonic map, minimizing the fractional Dirichlet energy in a different homotopy class. This is based on the study of the degree of fractional Sobolev maps and a sharp estimate \`a la Brezis-Coron. We give examples showing that it is in general not possible to minimize in every homotopy class and show a contrast with the 2-dimensional case.
\end{abstract}

\section{Introduction}
Minimizing fractional Sobolev energies of maps between manifolds, as pioneered e.g. in \cite{mirpis}, leads naturally to the notion of (fractional) $s$-harmonic maps ($0<s<1$). The model case $s=\frac12$ (corresponding to the so called half-harmonic maps) was first studied by Da Lio and Rivi\`ere \cite{DLR,DLR2} as critical points of the $L^2$-norm of the $\tfrac14$-Laplacian. In this paper we will focus on the special case of half-harmonic maps from $\R$ into the circle $\sph^1$. More precisely, consider the space $\dot H^{\frac12,2}(\R, \mathbb{C})$ of measurable maps $u$ from $\R$ into the complex plane  such that 
$$[u]_{H^{\frac{1}{2},2}(\R)}^2:= \int_{\R}\int_{\R}\frac{|u(x)-u(\xi)|^2}{(x-\xi)^2}dxd\xi<\infty,$$
and its non-convex subset
$$\dot H^{\frac12,2}(\R, \sph^1) = \left\{u\in \dot H^{\frac12,2}(\R, \mathbb{C}): |u(x)|=1 \text{ for a.e. }x\in\R\right\},$$
endowed with the same seminorm. 

For such maps, following \cite{DLR}, we introduce the fractional energy
\begin{equation}\label{12energy}
E(u):=\int_{\R}\left|(-\Delta)^\frac14u\right|^2dx, \quad u\in \dot H^{\frac12,2}(\R, \sph^1).
\end{equation}
This energy can be characterized in different equivalent ways, as we shall briefly recall. First of all, for $ u\in \dot H^{\frac12,2}(\R, \mathbb{C})$, its Poisson harmonic extension is well defined on $\R^2_+:=\R\times(0,\infty)$ as
$$\tilde u(x,y):=\frac{1}{\pi}\int_{\R}\frac{yu(\xi)}{y^2+(x-\xi)^2}d\xi.$$
Notice that if $u$ takes values into $\sph^1$, $\tilde u$ takes values into the closure of the unit disk
$$D^2:=\{(x,y)\in \R^2:x^2+y^2<1\}.$$
Then we have (see the Appendix, \eqref{FL2} and \eqref{equivnorms2} in particular) 
\begin{equation}\label{equiv1}
\|(-\Delta)^\frac14 u\|_{L^2(\R)}^2= \int_{\R^2_+}|\nabla \tilde u|^2 dxdy =\frac{1}{2\pi}[u]_{H^{\frac{1}{2},2}(\R)}^2,\quad u\in \dot H^{\frac12,2}(\R, \mathbb{C}),
\end{equation}
and we can equivalently use any of these formulas to define $E(u)$. 
While the characterizations of $E(u)$ in terms of the extension $\tilde u$ or in terms of a double integral are more convenient to do estimates, the original definition of $E(u)$ as in \eqref{12energy} makes the analogy with harmonic maps more evident. Indeed, interpreting $(-\Delta)^\frac14 u$ as a half-derivative of $u$, $E(u)$ can be seen as the fractional analog for functions in $\dot H^{\frac{1}{2},2}(\R,\sph^1)$ of the classical Dirichlet integral
$$E_2(u):=\int_{\Omega}|\nabla u|^2 dxdy,\quad u\in H^{1,2}(\Omega,M),$$
where $\Omega\subset\R^2$, $M\subset\R^N$ is an embedded Riemannian manifold, and
$$H^{1,2}(\Omega,M):=\left\{u\in H^{1,2}(\Omega,\R^N): u(x)\in M \text{ for a.e. }x\in \Omega\right\}.$$
In this analogy,
an important property that $E$ shares with $E_2$, and that will play an important role is the conformal invariance:  
$$E(u)=E(u\circ \phi)\quad \text{for } u\in \dot H^{\frac12,2}(\R, \sph^1),$$
where $\phi:\R\to\R$ is the restriction of a conformal map sending $\R^2_+$ into itself and its boundary to itself.

Concerning $E_2$, one of the model problems in the calculus of variations is the study of its minimizers or, more generally, its critical points under suitable constraints, leading to harmonic maps (harmonic functions, in the special case $M=\R$). In particular, the following uniqueness question was raised by Giaquinta and Hildebrandt \cite{GH}: Consider $M=\sph^2$, $\Omega=D^2\subset\R^2$, $g\in C^\infty(\de D^2,\sph^2)$, and let $u_0$ be a minimizer of $E_2$ in
$$H^{1,2}_g(D^2, \sph^2):=\{u\in H^{1,2}(D^2, \sph^2): u|_{\partial D^2}=g\}.$$
Assume that $g$ is nonconstant. Does $E_2$ necessarily have a second critical point $\bar u\ne u_0$? The question was answered in the affirmative by Brezis and Coron, and independently by Jost \cite{jost}:

\begin{thmx}[Brezis-Coron \cite{bc}, Jost \cite{jost}]\label{thmBC} If $g\in C^\infty(\partial D^2,\sph^2)$ is not constant, then $E_2$ has at least two critical points in $H^{1,2}_g(D^2,\sph^2)$.
\end{thmx}

The first aim of this paper is to show that a similar result holds for the fractional Dirichlet energy $E$.
More precisely, let $I=(-1,1)$ and, for a given $g\in \dot H^{\frac12,2}(\R, \sph^1)$ ,set
$$\mathcal{E}_g= \left\{u\in \dot H^{\frac12,2}(\R, \sph^1):u=g\text{ in }\R\setminus I\right\}.$$
By direct methods it is straightforward to see that there exists $u_0\in \M{E}_g$ such that
$$E(u_0)=\min_{\M{E}_g}E.$$
By a first variation argument (see e.g. \cite[Sec. 1]{milsir}), one sees that $u_0$ is a half-harmonic map, i.e. it satisfies
\begin{equation}\label{eq12harm}
(-\Delta)^\frac12 u_0 \perp T_{u_0}\sph^1\quad \text{in } I.
\end{equation}
Moreover, any solution to \eqref{eq12harm} is smooth in $I$, as shown by Da Lio and Rivi\`ere \cite{DLR} (see also \cite{MazSch0} for a more recent proof and \cite{DLS}, \cite{milpeg}, \cite{milpegsch}, \cite{milsir}, \cite{milsiryu} for the rapidly developing regularity theory of fractional harmonic maps).

We now want to show that, unless $g$ is constant in $\R\setminus I$, such minimizer is never unique (for the case $g$ constant, see Theorem \ref{thmconst} below).

\begin{thm}\label{mainthm} Let $g\in \dot H^{\frac12,2}(\R, \sph^1)$ be non-constant in $\R\setminus I$, and let $u_0\in \mathcal{E}_g$ be a minimizer of the functional $E:\mathcal{E}_g\to \R$. Then $E$ has a second critical point $u^*\ne u_0$.
\end{thm}

In order to prove Theorem \ref{mainthm} we shall use the same strategy of  Brezis-Coron \cite{bc} of minimizing in homotopy classes, although this will lead us to a geometrically different approach based on holomorphic maps. 

We start by defining the Brower degree of $u\in \dot H^{\frac12,2}(\R, \sph^1)$ as
\begin{equation}\label{deguPi}
\deg(u):=\deg(u\circ\Pi_+),
\end{equation}
where
\begin{equation}\label{defpi+}
\Pi_+:\sph^1\setminus\{i\}\to\R,\quad \Pi_+(x+iy)=\frac{x}{1-y}
\end{equation}
is the stereographic projection of $\sph^1\subset\mathbb{C}$ from the North pole $i$ and $\deg(u\circ\Pi_+)$ is the Brouwer degree of a map in $H^{\frac{1}{2},2}(\sph^1,\sph^1)$, see e.g. \cite{BN} and Section \ref{a:deg}.\footnote{Of course, if we used the stereographic projection from the South pole
$$ \Pi_-(x+iy)= \frac{x}{1+y}$$
we would obtain the opposite degree. With our choice, if $u(x)$ covers $\sph^1$ once in the counter-clockwise direction as $x$ goes from $-\infty$ to $+\infty$, then $\deg(u)=1$.}

Notice that $\Pi_+$ can be extended to a biholomorphic map $\Psi$ from $\cl{D^2}\setminus \{i\}$ onto $\bar \R^2_+$  (we identify $\R^2$ with $\mathbb{C}$), with inverse $\Phi$, see \eqref{Phi}-\eqref{Psi} below. Because of the conformal invariance of the Dirichlet integral we have
$$E(u)=\int_{\R^2_+}|\nabla \tilde u|^2 dx dy=\int_{D^2}|\nabla (\tilde u\circ \Psi) |^2 dxdy=:E(u\circ \Pi_+),$$ 
and under these trasformations we can identify $\dot H^{\frac12,2}(\R, \sph^1)$ and $H^{\frac12,2}(\sph^1, \sph^1)$. In fact, all the results of this paper can be stated equivalently for maps on $\R$ or on $\sph^1$.
With this in mind, \eqref{deguPi} is well defined for $u \in \dot H^{\frac12,2}(\R, \sph^1)$. 

Now define
$$\mathcal{E}_{g,k}=\{u\in \mathcal{E}_g: \deg(u)= k\}.$$
Then Theorem \ref{mainthm} follows immediately from the following:

\begin{thm}\label{mainthm2} Let $u_0$ minimize $E$ in $\mathcal{E}_g$ for some $g\in \dot H^{\frac12,2}(\R, \sph^1)$ non-constant in $\R\setminus I$. Then there exists $x_0\in I$ such that $u_0'(x_0)\ne 0$, and setting $k:=\deg(u_0)-\mathrm{sign}(u_0(x_0)\wedge u_0'(x_0))$,\footnote{for two vectors $v,w\in \R^2\simeq \mathbb{C}$ we define $v\wedge w:=v_1w_2-v_2w_1$.} $E$ has a minimizer $u^*$ in $\M{E}_{g,k}$. 
\end{thm}

The main difficulty in Theorem \ref{mainthm2}, as already shown in \cite{mirpis}, is that the weak convergence in $H^{\frac12,2}$ does not preserve the degree, so that a minimizing sequence in $\M{E}_{g,k}$ can concentrate at some points, and  weakly converge to a function in $\M{E}_{g,\ell}$ with $\ell\ne k$, see also Proposition  \ref{p:jump}.

The choice of $k$ in Theorem \ref{mainthm2} has the following geometric meaning: If $u_0$ at $x_0$ is moving along $\sph^1$ in a given direction, then the degree of $u^*$  is obtained by turning around $\sph^1$ once in the opposite direction. More precisely, in Lemma \ref{strictestimate} we shall glue near $x_0$ a concentrated circle with orientation opposite to $u'(x_0)$. This choice of orientation is necessary since, in general we cannot expect $E$ to have a minimizer in $\M{E}_{g,k}$ for every $k\in\mathbb{Z}$, as the following example shows.

\begin{thm}\label{thmex}
Consider for $\lambda>0$ the map $g_\lambda:\R\to \sph^1$ given by
\begin{equation}\label{defglambda}
g_\lambda(x)=\Pi_+^{-1}\left(\frac{x}{\lambda}\right),
\end{equation}
where $\Pi_+$ is as in \eqref{defpi+}.
Then:
\begin{enumerate}
\item[(i)] for every $\lambda>0$, $u_0=g_\lambda$ is the only minimizer of $E$ in $\M{E}_{g_\lambda,1}$;
\item[(ii)] there exists $\lambda^*>1$ such that for $\lambda\ge \lambda^*$, $u_0=g_\lambda$ is the only minimizer of $E$ in $\M{E}_{g_\lambda}$;
\item[(iii)] for every $\lambda>0$, $E$ has a minimizer $u^*$ in $\M{E}_{g_\lambda,0}$, and there exists $\lambda_0>1$ such that for $\lambda\in (0,\lambda_0)$, $u^*$ is an absolute mimimizer, i.e. it mimimizes $E$ in $\M{E}_{g_\lambda}$, while $g_\lambda$ does not minimize $E$ in $\M{E}_{g_\lambda}$;
\item[(iv)] for every $\lambda>0$, $E$ has no minimizer in $\M{E}_{g_\lambda,k}$ for $k>1$.
\end{enumerate}
\end{thm}


Theorem \ref{thmex} can be seen as a $1$-dimensional counterpart of Theorem 2 in \cite{bc}, where for the boundary datum 
\begin{equation}\label{gGH}
g_R:\partial D^2\to \sph^2,\quad g_R(x,y):=(Rx,Ry,\sqrt{1-R^2}), \quad 0<R<1,
\end{equation}
it was shown that:
\begin{enumerate}
\item the absolute minimizer of $E_2(u)=\int_{D^2}|\nabla u|^2dxdy$ in $H^{1,2}_{g_R}(D^2,\sph^2)$ is
$$\underline{u}(x,y)=\frac{2\lambda}{\lambda^2+x^2+y^2}(x,y,\lambda)+(0,0,-1),\quad \lambda:=\frac{1}{R}+\sqrt{\frac{1}{R^2}-1} ;$$
\item the function
$$\overline{u}(x,y)=\frac{2\mu}{\mu^2+x^2+y^2}(x,y,-\mu)+(0,0,1),\quad \mu:=\frac{1}{R}-\sqrt{\frac{1}{R^2}-1}$$
minimizes in the homotopy class obtained by lowering the Brouwer degree (suitably definied) by $1$, hence it is a second harmonic map in $H^{1,2}_{g_R}(D^2,\sph^2)$;
\item there is no minimizer in \emph{any} other homotopy class.
\end{enumerate}

The analogy between Theorem \ref{thmex} and Theorem 2 of \cite{bc} is very strong if one considers that the functions $\overline u$ and $\underline u$ above (and their restriction $g_R$ in \eqref{gGH}) are rescalings of the inverse stereographic projection from the South and North pole respectively. 

On the other hand, in contrast with \cite[Thm. 2]{bc}, 
 in Theorem \ref{thmex} (iii), for $\lambda$ slightly bigger than $1$ (which would be the analog of $R$ slightly smaller than $1$ in \eqref{gGH}) we do not have that $u_0=g_\lambda$ is the absolute minimizer of $E$ in $\M{E}_{g_\lambda}$, as
$$\inf_{\M{E}_{g_\lambda,0}} E < E(g_\lambda)=\min_{\M{E}_{g_\lambda,1}}E,$$
see Proposition \ref{propex}, so that in part (ii) of Theorem \ref{thmex} one cannot take $\lambda^*=1$. 
Moreover, it is open whether in part (iv) of Theorem \ref{thmex}, $E$ has minimizers in $\M{E}_{g_\lambda,k}$ for $k<0$.

The first of these facts might sound a bit surprising at first, since for $\lambda>1$, $g_\lambda|_{\R\setminus I}$ covers more than half of the circle, hence any map in $\M{E}_{g_\lambda,0}$ must cover more than half the circle at least \emph{twice} (in opposite directions), while maps in $\M{E}_{g_\lambda,1}$ can cover the whole circle only once. The non-locality of the energy, when $\lambda$ is not much bigger than $1$ makes the first option more convenient. Of course this is not the case with the local energy $\int_{\R}| u'|dx$, as in this case we have 
$$\inf_{u\in \M{E}_{g_\lambda,0}}\int_{\R}|u'|dx>2\pi =\inf_{u\in \M{E}_{g_\lambda,1}}\int_{\R}|u'|dx =\int_{\R}|g_\lambda'|dx\quad \text{for }\lambda>1.\footnote{These integrals are set to be $+\infty$ when $u'$ is not integrable.}$$

Also the second contrast between Theorem 2 of \cite{bc} and Theorem \ref{thmex} (part iv) above arises from the nonlocal nature of the fractional energy $E$. Indeed the method used by Brezis and Coron of gluing functions $u_1$ and $u_2$ via a Kelvin transformation and expressing $E_2$ of the glued functions as $E_2(u_1)+E_2(u_2)$, does not apply to the energy $E$.

\medskip

We also address the case when $g\equiv P\in \sph^1$ is constant. While it is easy to see that in this case $E$ has no minimizer in $\mathcal{E}_{g,k}$ for any $k\ne 0$ and $u\equiv P$ is the only minimizer in $\mathcal{E}_{g,0}$ (see Lemma \ref{l:Blas} below), a priori there could be other critical points. We will show that this is not the case. Because of its independent interest, we shall state the result for general target manifolds.

\begin{thm}\label{thmconst} Given a smooth closed manifold $\mathcal{N}$ isometrically embedded in $\R^N$ and $$u\in \dot H^{\frac12,2}(\R,\mathcal{N}):=\left\{u\in \dot H^{\frac12,2}(\R, \R^N): u(x)\in \mathcal{N} \text{ for a.e. }x\in\R\right\},$$
if $u(x)\equiv P$ in $\R\setminus I$ for some $P\in \mathcal{N}$ and $u$ is half-harmonic in $I$, i.e.
\begin{equation}\label{eq12harmN}
(-\Delta)^\frac{1}{2} u(x):=-\frac{\de \tilde u(x,0)}{\de y} \perp T_{u(x)} \mathcal{N} \quad \text{for }x\in I,
\end{equation}
where $\tilde u:\R^2_+\to \R^N$ is the Poisson harmonic extension, then $u\equiv P$ in $\R$.
\end{thm}

While \eqref{eq12harmN} should be initially read in a weak sense, by the regularity result of Da Lio-Rivi\`ere \cite{DLR2} (see also \cite{MazSch0} for a more recent proof), any solution to \eqref{eq12harmN} is smooth in $I$.

Theorem \ref{thmconst} is the fractional analog of a well-known result of Lemaire \cite{Lem}: a harmonic map from the disk $D^2\subset\R^2$ into a closed smooth manifodl $\mathcal{N}\subset\R^N$ which is constant on $\partial D^2$, must be constant everywhere. We remark that Lemaire crucially uses that $D^2$ is simply connected and gives counterexamples on the annulus. It is open whether Theorem \ref{thmconst} holds true or fails on unions of intervals.

\medskip

Finally, we mention that, similar to Theorem \ref{thmBC} and Theorems \ref{mainthm} and \ref{mainthm2}, Isobe and Marini \cite{IsoMar} proved the existence of $SU(2)$ Yang-Mills connection on $B^4\subset\R^4$ having the same boundary value, but different Chern classes. Their work is based on a strict estimate of Taubes \cite{Tau}, in the same spirit of \cite[Lemma 2]{bc} and Lemma \ref{strictestimate} below. See \cite{MarRiv} for a new approach to the strict inequality, related topics and open problems.

\medskip

\noindent\textbf{Acknowledgements} The second author first learnt about degree theory in $H^{\frac12,2}(\sph^1,\sph^1)$ (more generally in $VMO(\sph^1,\sph^1)$) and the problem of minimizing in different homotopy classes during a very inspiring talk of H. Brezis in honour of G. Stampacchia held in Pisa in November 2010. He was later made aware of the potential applications in the study of the regularity theory of $2$-dimensional half-harmonic maps by T. Rivi\`ere. This will be developed in future works. We are also grateful to A. Pisante for useful comments. We would like to warmly thank the anonymous referee for careful reading and for several useful remarks.

\section{Degree for maps in $H^{\frac{1}{2},2}(\sph^1,\sph^1)$ and $\dot H^{\frac{1}{2},2}(\R,\sph^1)$}\label{a:deg}

We first recall some facts about the degree of maps in $H^{\frac{1}{2},2}(\sph^1,\sph^1)$ and $\dot H^{\frac{1}{2},2}(\R,\sph^1)$. To our knowledge, this was first introduced by L. Boutet de Monvel and O. Gabber back in 1984, see the discussion in \cite{BGP}. For a classical introduction to the topic we refer to \cite{BN}.

Given a function $u\in C^0(\sph^1,\sph^1)$, one typically defines its Brouwer degree via a lifting, i.e. a map $\varphi\in C^0(\R,\R)$ such that
\begin{equation}\label{liftu}
u(e^{i\theta})= e^{i\varphi(\theta)}.
\end{equation}
Since $u(e^{i\theta})=u(e^{i(\theta+2\pi)})$, it follows that $\varphi(2\pi)-\varphi(0)\in 2\pi\mathbb{Z}$, and one defines
\begin{equation}\label{degC0}
\deg(u):=\frac{\varphi(2\pi)-\varphi(0)}{2\pi}\in\mathbb{Z}.
\end{equation}
On the other hand, due to the nature of our problems, it will be more natural to work with integro-differential definitions of the Brouwer degree.
We start recalling that for $u\in C^\infty(\sph^1,\sph^1)$ its Brouwer degree can be computed via the de Rahm cohomology as 
\begin{equation}\label{degformula0}
\deg(u)=\frac{1}{2\pi}\int_{\sph^1} u^*(d\theta)=\frac{1}{2\pi}\int_{\sph^1}u\wedge\dot u\,d\theta=\frac{1}{2\pi i}\int_{\sph^1}\bar u \dot u\,d\theta.\footnote{With a slight abuse of notation, when integrating we identify $\sph^1$ with $\R/\mathbb{Z}$, so that for $u:\sph^1\subset\mathbb{C}\to \mathbb{C}$ we have $\int_{\sph^1} u d\theta=\int_{\sph^1}u(e^{i\theta})d\theta:=\int_{\theta_0}^{\theta_0+2\pi} u(e^{i\theta})d\theta$ for any $\theta_0\in\R$. Moreover, $f'$ will denote derivatives on $\R$ and $\dot f$ will denote derivatives of $S^1$, i.e. $\dot f(e^{i\theta})= \frac{df(e^{i\theta})}{d\theta}$.}
\end{equation}  Here we used that for vectors $v,w\in \R^2\simeq\mathbb{C}$, $\bar v w=v\cdot w+i( v\wedge w)$, where $\bar v w=(v_1-iv_2)(w_1+iw_2)$ is the usual product in $\mathbb{C}$, while $v\wedge w:=(v_1w_2-v_2w_1)$ and $v\cdot w:=v_1w_1+v_2 w_2$, and we used that $u\cdot \dot u=0$.

This definition is equivalent to the previous one, since for $\varphi\in C^\infty(\R)$ lifting $u$ in the sense of \eqref{liftu}, we have
$$\frac{1}{i}\int_{\sph^1}\bar u \dot u\,d\theta =\int_0^{2\pi} \varphi'(x)\,dx={\varphi(2\pi)-\varphi(0)}.$$
Further, writing $u$ as Fourier series
$$u(e^{i\theta})=\frac{1}{\sqrt{2\pi}}\sum_{k\in\mathbb{Z}}a_k e^{ik\theta},\quad a_k:=\frac{1}{\sqrt{2\pi}}\int_{\sph^1}u(e^{i\theta})e^{-ik\theta}d\theta,$$
 with a simple computation we obtain (still assuming $u$ to be smooth)
\begin{equation}\label{degformula01}
\frac{1}{i}\int_{\sph^1} \bar u \dot u\, d\theta=\sum_{k\in\mathbb{Z}}k|a_k|^2.
\end{equation}
While the integral in \eqref{degformula01} is naturally defined for $u\in C^1(\sph^1,\sph^1)$, the sum on the right-hand side is defined for any $u\in H^{\frac12,2}(\sph^1,\sph^1)$,\footnote{Actually, also the integral in \eqref{degformula01} can be defined for $u\in H^{\frac{1}{2},2}(\sph^1,\sph^1)$ using the duality between $H^{-\frac{1}{2},2}$ and $H^{\frac{1}{2},2}$, as observed by Boutet de Monvel and Gabber (see \cite[p. 224]{BN} and \cite[Sec. 3]{bre}).} (see the Appendix for the definitions of fractional Sobolev spaces and their seminorms), so that we can also define
\begin{equation}\label{degformula}
\deg(u):=\frac{1}{2\pi}\sum_{k\in\mathbb{Z}}k|a_k|^2, \quad u\in H^{\frac{1}{2},2}(\sph^1,\sph^1).
\end{equation}
It follows at once that
\begin{equation}\label{degest}
|\deg(u)|\le \frac{1}{2\pi}\sum_{k\in\mathbb{Z}}|k||a_k|^2 =\frac{1}{2\pi}[u]_{H^{\frac{1}{2},2}}^2.
\end{equation}
Moreover, with the Hilbert transform
$$\mathcal{H}(u)=\sum_{k\in\mathbb{Z}\setminus\{0\}}\frac{-i k}{|k|}a_ke^{ik\theta}, \quad u\in H^{\frac{1}{2},2}(\sph^1,\sph^1),$$
we can write
\begin{equation}\label{degformulaH}
\begin{split}
\sum_{k\in\mathbb{Z}} k|a_k|^2
&=\frac{1}{2\pi} \int_{\sph^1} \sum_{k\in\mathbb{Z}\setminus\{0\}} \frac{k}{|k|}\sqrt{|k|} \bar a_ke^{-ik\theta} \sum_{\ell\in\mathbb{Z}\setminus\{0\}}\sqrt{|l|}a_\ell e^{i\ell \theta}d\theta\\
&={-i} \int_{\sph^1} \mathcal{H}[(-\Delta)^{\frac14} \bar u] (-\Delta)^{\frac14} u d\theta,
\end{split}
\end{equation}
(see \eqref{lapl14} for the definition of $(-\Delta)^{\frac14}$) hence, from \eqref{degformula}, we have the equivalent definition
\begin{equation}\label{degformulaH}
\deg (u):= \frac{1}{2\pi i}\int_{\sph^1} \mathcal{H}[(-\Delta)^{\frac14} \bar u] (-\Delta)^{\frac14} u d\theta,  \quad u\in H^{\frac{1}{2},2}(\sph^1,\sph^1).
\end{equation}
Going back to \eqref{degformula0}, for $u\in C^\infty(\sph^1,\sph^1)$, which we write as $u=u_1+iu_2$, considering a smooth extension $U=U_1+iU_2$ to the unit disk $\cl{D^2}$, and using Stokes' theorem we get
\begin{equation*}
\begin{split}
\frac{1}{i }\int_{\sph^1} \bar u \dot u d\theta&=\frac{1}{i}\int_{\sph^1}\bar u d u=\frac{1}{ i} \int_{D^2} d(\bar U dU) =\frac{1}{ i} \int_{D^2} d[(U_1-iU_2)(dU_1+idU_2)]\\
&=2\int_{D^2} dU_1\wedge dU_2=2\int_{D^2}\(\frac{\partial U_1}{\partial x} \frac{\partial U_2}{\partial y}-\frac{\partial U_1}{\partial y}\frac{\partial U_2}{\partial x}\)dxdy=2\int_{D^2} JU dxdy.
\end{split}
\end{equation*}
Considering that $|JU|\le \tfrac{1}{2}|\nabla U|^2$, the right-hand side is well defined for any $u\in H^{\frac{1}{2},2}(\sph^1,\sph^1)$ choosing as $U$ any $H^1$ extension, for instance the harmonic extension $\tilde u\in H^{1,2}(D^2,\mathbb{C})$. Hence we can also define
\begin{equation}\label{degformula2}
\deg(u):=\frac{1}{\pi}\int_D JU dxdy,\quad u\in H^{\frac{1}{2},2}(\sph^1,\sph^1), \, U\in H^{1,2}(D^2,\mathbb{C}), \, U|_{\partial D^2}=u.
\end{equation}
In a similar fashion, if $u\in C^\infty(\R,\sph^1)$ is constant outside a bounded interval, say $u\equiv a$ in  $\R\setminus [-R,R]$, we can write
\begin{align}\label{46}\deg(u)=\frac{1}{2\pi}\int_{\R}u^*(d\theta)=\frac{1}{2\pi} \int_{\R} u\wedge u' dx= \frac{1}{2\pi i}\int_\R \bar u u' dx.\end{align}
Now, writing $w=u-a$, we have that $w$ is compactly supported and $\int_\R \bar u u' dx=\int_\R \bar w w' dx$.
Then, we can write  $w'$ as Fourier integral and compute   
 \begin{align}\label{degintR} 
\deg(u)&=\frac{1}{2\pi i}\int_\R \bar w w' dx\notag=\frac{1}{2\pi i}\int_\R \bar w(x) \left(\frac{i}{\sqrt{2\pi}}\int_\R e^{i\xi x} \xi \hat w(\xi) d\xi  \right)dx \notag\\&=\frac{1}{2\pi }\int_\R \xi\hat w(\xi) \left(\frac{1}{\sqrt{2\pi}}\int_\R e^{i\xi x} \bar w(x) dx  \right)d\xi \notag =\frac{1}{2\pi }\int_\R \xi\hat w(\xi) \check{\bar w}(\xi) d\xi\notag \\ &=\frac{1}{2\pi}\int_{\R}\xi |\hat w(\xi)|^2 d\xi,
\end{align} where in the last equality we have used that $\bar {\hat w}=\check{\bar w}$. 
On the other hand
 \begin{align}  \label{degintR2}
 \int_{\R}\mathcal{H}[(-\Delta)^\frac{1}{4}\bar w]& (-\Delta)^\frac{1}{4} w dx=\int_{\R}\mathcal{F}^{-1}\left(\frac{-i \xi}{|\xi|}\mathcal F((-\Delta)^\frac{1}{4}\bar w)\right) (-\Delta)^\frac{1}{4} w dx\notag \\
 &=-i\int_{\R}\frac{\xi}{|\xi|}\mathcal F((-\Delta)^\frac{1}{4}\bar w) \mathcal{F}^{-1}\left((-\Delta)^\frac{1}{4} w\right) d\xi= -i\int_\R\frac{\xi}{|\xi|}\sqrt{|\xi|}\hat{\bar w}(\xi)\sqrt{ |\xi|}\check w(\xi) d\xi\notag\\&=-i\int_\R \xi \bar{\check w}(\xi) \check w(\xi) d\xi =-i\int_\R \xi|\check w(\xi)|^2d\xi=i\int_\R \xi|\hat w(\xi)|^2d\xi ,\end{align} 
where $(-\Delta)^\frac{1}{4}$ is defined in \eqref{fraclapl2} and $\M{H}$ denotes the Hilbert transform:  \begin{align}\label{def-Hilbert-R} \mathcal H ({v}):=\mathcal F^{-1}\left(\frac{-i \xi}{|\xi|}\hat {{v}}\right) . \end{align}
Putting \eqref{degintR} and \eqref{degintR2} together yields
\begin{equation}\label{degintR3}
\deg(u)=\frac{1}{2\pi i}\int_{\R}\mathcal{H}[(-\Delta)^\frac{1}{4}{\bar w}] (-\Delta)^\frac{1}{4} {w} dx=\frac{1}{2\pi i}\int_{\R}\mathcal{H}[(-\Delta)^\frac{1}{4}\bar u] (-\Delta)^\frac{1}{4} u dx,
\end{equation}
which, using Lemma \ref{l:approx}, can be extended by density to every $u\in \dot H^{\frac12,2}(\R,\sph^1)$.
It follows from \eqref{degintR3} and H\"older's inequality
\begin{equation}\label{stimadegR}
|\deg(u)|\le \frac{1}{2\pi}\|(-\Delta)^\frac14 u\|_{L^2}^2=\frac{1}{2\pi}E(u).
\end{equation}
Finally we remark that by composing with the conformal map $\Psi(z)=-i\frac{z+i}{z-i}$, sending $D^2$ onto $\R^2_+$ (see \eqref{Psi}), one infers from \eqref{degformula2} also
\begin{equation}\label{degformula4}
\deg(u)=\frac{1}{\pi}\int_{\R^2_+}JUdxdy,\quad U\in H^{1,2}(\R^2_+), \, U|_{\R\times\{0\}}=u.
\end{equation}

\begin{prop}\label{p:equivdeg}
Given $u\in H^{\frac12,2}(\sph^1,\sph^1)$ (resp. $u\in \dot H^{\frac12,2}(\R,\sph^1)$), the definitions of Brouwer degree given by \eqref{degformula}, \eqref{degformulaH} and \eqref{degformula2} (resp. \eqref{degintR}, \eqref{degintR3} and \eqref{degformula4}) are equivalent and take integer values. Moreover, given $u\in \dot H^{\frac12,2}(\R,\sph^1)$, $\deg(u)=\deg(u\circ\Pi_+)$, where $\Pi_+: \sph^1\setminus\{i\}\to \R$ is the stereographic projection as in \eqref{defpi+}. 
\end{prop}

\begin{proof}
The equivalence of the above formulas for $u\in C^1(\sph^1,\sph^1)$ or for $u\in C^1(\R,\sph^1)$ constant outside an interval follows from the above computations and is an integer by classical theory. When $u\in H^{\frac{1}{2},2}(\sph^1,\sph^1)$ or $u\in \dot H^{\frac12,2}(\R,\sph^1)$, this follows by approximation using Propositions \ref{l:approx}, \ref{l:approx2} and \ref{l:approx3}.
\end{proof}

Finally, one might wonder whether for $u\in C^0\cap H^{\frac{1}{2},2}(\sph^1,\sph^1)$ the degree defined in \eqref{degC0} coincides with the one defined in  \eqref{degformula}, \eqref{degformulaH} and \eqref{degformula2}. That this is the case follows essentially from Lemma \ref{l:approx0}, since the averaging functions $u_{\av, \ve}$ converge uniformly to $u$ as $\ve\to 0$ for $u\in C^0(\sph^1,\sph^1)$, hence $\deg(u_{\av, \ve})=\deg(u)$ for $\ve$ small.

\section{Degree jumps and surgeries}

\subsection{Degree jumps under weak convergence}

The first type of degree jump that we will consider is due to the fact that the degree is not continuous with respect to the weak convergence in $H^{\frac12,2}(\sph^1,\sph^1)$, or in $\dot H^{\frac12,2}(\R,\sph^1)$.

We shall say that a sequence $(u_n)\subset \dot  H^{\frac{1}{2},2}(\R,\sph^1)$ 
 weakly converges to $u$ in $\dot H^{\frac{1}{2},2}(\R,\sph^1)$ if
$$\lim_{n\to \infty}\int_{\R}(-\Delta)^\frac14 u_n  (-\Delta)^\frac14 v dx=\int_{\R}(-\Delta)^\frac14 u  (-\Delta)^\frac14 v dx,\quad \text{for } v\in \dot H^{\frac{1}{2},2}(\R,\mathbb{C}).$$
As already mentioned, weakly converging sequences might jump from a homotopy class to another. The next proposition, which is the analog \cite[Lemma 3]{kuw},  and (15) in \cite{bc} (see also \cite{jost} and \cite[Lemma 1]{bc2}), quantifies the minimal energy cost of such jumps.

\begin{prop}\label{p:jump} Let $(u_n)\subset \mathcal{E}_k=\mathcal{E}_k(\R,\sph^1):=\{v\in \dot H^{\frac12,2}(\R,\sph^1):\deg(v)=k\}$, converge weakly in $\dot H^{\frac{1}{2},2}(\R,\sph^1)$ to $u\in \mathcal{E}_{\ell}$. Then $$E(u)\le \liminf_{n\to\infty} E(u_n)-2\pi |k-\ell|.$$
Equivalently, if $(u_n)\subset \mathcal{E}_k=\mathcal{E}_k(\sph^1,\sph^1):=\{v\in  H^{\frac12,2}(\sph^1,\sph^1):\deg(v)=k\}$ converge weakly in $\dot H^{\frac{1}{2},2}(\sph^1,\sph^1)$ to $u\in \mathcal{E}_{\ell}$, then $$E(u)\le \liminf_{n\to\infty} E(u_n)-2\pi |k-\ell|.$$ 
\end{prop}

In the proof of Proposition \ref{p:jump} we will use the following:

\begin{lem}\label{Lem-Hilbert} For $v,w\in H^{\frac12,2}(\sph^1,\mathbb{C})$, we have $\mathcal H((-\D)^\frac14 w) =(-\D)^\frac14(\mathcal H (w) )$ and
\begin{equation}\label{eqlemH}
\int_{\sph^1} \mathcal H((-\D)^\frac14 \bar v) (-\D)^\frac14w d\theta=-\overline{\int_{\sph^1}\mathcal (-\D)^\frac14  v (-\D)^\frac14 (\mathcal H (\bar w))d\theta}.
\end{equation}
\end{lem}
\begin{proof} The first identity follows easily from the definition of the Hilbert transform and $(-\Delta)^\frac14$ on $H^{\frac12,2}(\sph^1,\mathbb C)$. From this, and from the identities
$$\overline{\mathcal{H}(f)}=\mathcal{H}(\bar f),\quad \int_{\sph^1}\mathcal{H}(\bar f) gd\theta = -\int_{\sph^1}\bar f\mathcal{H}(g) d\theta,\quad f,g\in L^2(\sph^1),$$
we immediately get \eqref{eqlemH}. 
\end{proof}

\noindent\emph{Proof of Proposition \ref{p:jump}}
We prove the proposition for $(u_n)\subset \mathcal E_k(\sph^1,\sph^1)$. We define
$$F_\pm(u)=E(u)\pm 2\pi \deg (u),$$
and prove that $F_\pm$ is weakly lower semicontinuous in $ H^{\frac{1}{2},2}(\sph^1,\sph^1)$.
It will be convenient to use the degree formula \eqref{degformulaH}
Set $w_n:= u_n-u\rightharpoonup 0$. Using \eqref{degformulaH}, and Lemma \ref{Lem-Hilbert} we get
\begin{align*}
F_\pm (u_n)&=E(u)+E(w_n)+o(1) \pm \frac{1}{i}\int_{\sph^1}\mathcal{H}[(-\Delta)^\frac14(\bar u+\bar w_n)](-\Delta)^\frac14(u+w_n)d\theta\\
&=F_\pm (u)+ E(w_n)  \pm \frac{1}{i} \int_{\sph^1}\mathcal{H}[(-\Delta)^\frac14\bar w_n](-\Delta)^\frac14 w_n d\theta+o(1),
\end{align*}   provided \begin{align}\label{zero}\int_{\sph^1}\mathcal{H}[(-\Delta)^\frac14\bar w_n](-\Delta)^\frac14 u d\theta\to0\quad\text{and }\int_{\sph^1}\mathcal{H}[(-\Delta)^\frac14\bar u ](-\Delta)^\frac14 w_n d\theta\to0.\end{align}
Then, using that
$$\left|\int_{\sph^1}\mathcal{H}[(-\Delta)^\frac14\bar w_n](-\Delta)^\frac14 w_n d\theta\right|\le  \|(-\Delta)^\frac14\bar w_n\|_{L^2}^2=E(w_n)$$
we obtain $F_\pm(u)\le F_\pm(u_n)+o(1)$ as $n\to \infty$, and we conclude that $F_\pm$ is weakly lower-semicontinuous. Then the lemma follows from
$$E(u)=F_{\pm}(u)\mp 2\pi \ell \le \liminf_{n\to\infty}F_{\pm}(u_n)\mp2\pi \ell=\liminf_{n\to\infty} E(u_n) \pm 2\pi(k-\ell).$$ 
It follows from the definition of weak convergence and Lemma \ref{Lem-Hilbert} that \eqref{zero} hold true, and we conclude the proof of the proposition. 
 \hfill $\square$

\begin{rmk}
The constant $2\pi$ is sharp, as can be proven by considering a concentrating sequence $(u_n)$ of  (restrictions of) M\"obius functions with $\deg (u_n)=1$, $E(u_n)=2\pi$, $u_n\rightharpoonup u\equiv C$ as $n\to\infty$. More generally, using (restriction of) Blaschke products as in \eqref{Blas}, one can construct a sequence $(u_n)$ with $\deg(u_n)=k>0$, $E(u_n)=2\pi k$ (by Lemma \ref{l:Blas}), $u_n\rightharpoonup u$, $\deg(u)=\ell\in [0,k)$, $E(u)=2\pi \ell$.
\end{rmk}

\begin{rmk}
One can give an analogous proof of Proposition \ref{p:jump} by working with $\tilde u_k\rightharpoonup \tilde u$ in $H^{1,2}(\R^2_+)$ and using \eqref{degformula4}, together with $|J\tilde u|\le \frac{1}{2}|\nabla \tilde u|^2$.
\end{rmk}

\begin{rmk}
In attempting to formulate a minimization problem in $\dot W^{s,\frac{1}{s}}(\R,\sph^1)$ or in $W^{s,\frac{1}{s}}(\sph^1,\sph^1)$ with $s\ne \frac12$, in the spirit of \cite{MazSch, MazSch2}, one of the main obstacles would be finding the analog of the sharp constant $2\pi$.
\end{rmk}

\subsection{Surgeries}

Next we will consider a degree change due to surgery, namely, given a closed arc $A\subset \sph^1$ we will  replace a function $u\in H^{\frac{1}{2},2}(\sph^1,\sph^1)$ in $A$ with another map $v:A\to \sph^1$ such that the resulting map is still in $H^{\frac12,2}(\sph^1,\sph^1)$. Up to a M\"obius transformation, we can assume $A=\sph^1_+:=\sph^1\cap \{(x,y)\in \R^2:y\ge 0\}$.

 Set the gluing of $v$ to $u$ as
\begin{equation}\label{defw}
v\& u(e^{i\theta}):=
\begin{cases}
v(e^{i\theta}) &\text{for }0\le \theta\le \pi\\
u(e^{i\theta}) &\text{for } \pi <\theta< 2\pi,
\end{cases}
\end{equation}
and the gluing of $v$ to the reflection of $u$ as
\begin{equation}\label{defuv}
v\# u(e^{i\theta}):=
\begin{cases}
v(e^{i\theta}) &\text{for }0\le \theta\le \pi\\
u(e^{-i\theta}) &\text{for } \pi <\theta< 2\pi.
\end{cases}
\end{equation}
If $u\in C^0(\sph^1,\sph^1)$, $v\in C^0(\sph^1_+,\sph^1)$ and $u(\pm1)=v(\pm1)$, so that $v\&u, v\# u\in C^0(\sph^1,\sph^1)$, it is fairly easy to see that 
$$\deg(v\& u)= \deg(u)+\deg (v\#u).$$
In particular the degree change only depends on the values of $u$ and its replacement in $\sph^1_+$, as it is well known.

We want to prove the same formula for $u,v\in H^{\frac12,2}(\sph^1,\sph^1)$. Two main difficulties arise. First of all we do not have an analog of $u(\pm1)=v(\pm1)$ in terms of traces of functions in $H^{\frac12,2}$. As Proposition \ref{example2} shows, and contrary to what happens for maps in $H^{1,2}(\Omega,\sph^2)$ with $\Omega\Subset \R^2$, a map $u\in H^{\frac12,2}(\sph^1,\sph^1)$ might not have a trace on the boundary of an arc $A$ of $\sph^1$, and in fact, for some point $e^{i\theta_0}\in \sph^1$, $u(e^{i\theta})$  can wind infinitely many times around $\sph^1$ as $\theta\to \theta_0^+$ and infinitely many times, but with opposite orientation as $\theta\to \theta_0^-$. The heavy cancelation due to the opposite orientation of the windings allow the degree of $u$ to still be finite, but this suggests extra care in cutting, gluing and counting the degree of functions in $H^{\frac12,2}(\sph^1,\sph^1)$.

This might appear to be only a problem of trace (in particular of the unboundedness of functions in $H^{\frac12,2}(\sph^1,\sph^1)$), but problems arise even if one considers functions in $H^{\frac12,2}\cap C^0(\sph^1,\sph^1)$, as Proposition \ref{example1} shows. In fact, one can take a map $H^{\frac12,2}\cap C^0(\sph^1,\sph^1)$ and replace it locally by a constant map obtaining a new map that, in spite of being continuous, does not belong to $H^{\frac12,2}(\sph^1,\sph^1)$.  Both phenomena are essentially due to the nonlocal nature of the $H^{\frac12,2}$-seminorm, and in particular to cancelations in the double integral defining it. Because of them, it will be necessary to assume that $v\&u, v\#u\in H^{\frac12,2}(\sph^1,\sph^1)$:

\begin{prop}\label{p:degjump} Consider $u \in H^{\frac{1}{2},2}(\sph^1,\sph^1)$ and $v:\sph^1_+\to \sph^1$. Set $v\&u$ and $v\#u$ as in \eqref{defw} and \eqref{defuv} and assume that $v\&u, v\#u\in H^{\frac{1}{2},2}(\sph^1,\sph^1)$. Then
\begin{equation}\label{eq:jump}
\deg(v\&u)= \deg(u)+\deg (v\#u).
\end{equation}
\end{prop}

\begin{proof}
Consider $\gamma:[-1,1]\times\{0\}\to\sph^1_-$ given by $\gamma(x,0)=e^{i\frac{\pi}{2}(x-1)}$. Clearly $\gamma$ is a diffeomorphism and
\begin{equation}\label{diff}
|\gamma'|\equiv\frac{\pi}{2},\quad |(\gamma^{-1})'|\equiv\frac{2}{\pi}.
\end{equation}
On $[-1,1]\times \{0\}$ define the function
$$u^*(x,0)=u(\gamma(x,0))=u(e^{i\frac{\pi}{2}(x-1)}).$$
Now define the maps $u^*_\pm :\partial D^2_\pm \to \sph^1$
\[
\begin{cases}
u^*_\pm :=u&\text{on }\sph^1_\pm\\
u^*_\pm =u^*&\text{on } (-1,1)\times \{0\},
\end{cases}
\]
and $v^*:\partial D^2_+\to \sph^1$
\[
\begin{cases}
v^*=v&\text{on }\sph^1_+\\
v^*=u^*&\text{on } (-1,1)\times \{0\}.
\end{cases}
\]
Since  $|(x,0)-(\xi,0)| \ge \frac{2}{\pi}|\gamma(x,0)-\gamma(\xi,0)|$   for $x, \xi\in [-1,1]$, also using \eqref{diff}, we get 
\begin{equation}\label{stimau*}
\begin{split}
\int_{-1}^{1}\int_{-1}^{1}\frac{|u^*(x,0)-u^*(\xi,0)|^2}{|(x,0)-(\xi,0)|^2}dxd\xi
&\le \frac{\pi^2}{4} \int_{-1}^{1}\int_{-1}^{1}\frac{|u(\gamma(x,0))-u(\gamma(\xi,0))|^2}{|\gamma(x,0)-\gamma(\xi,0)|^2}dxd\xi\\
&=  \int_{\sph^1_-}\int_{\sph^1_-}\frac{|u(e^{is})-u(e^{it})|^2}{|e^{is}-e^{it}|^2}dsdt.
\end{split}
\end{equation}
We now claim
\begin{align}\label{claimgamma0}
|e^{is}-(x,0)|&\ge \frac{|e^{is}-\gamma(x,0)|}{2},\quad \text{for }e^{is}\in \sph^1_-, \, x\in [-1,1],\\
\label{claimgamma}
|e^{is}-(x,0)|&\ge \frac{|e^{is}-\gamma(x,0)|}{4},\quad \text{for }e^{is}\in \sph^1_+,\ x\in [-1,1].
\end{align}
To prove \eqref{claimgamma0}, by symmetry it suffices to consider the case $-\frac{\pi}{2}\le s\le 0$. We consider $2$ cases.

\noindent\textbf{Case 1:} $0\le \frac{\pi}{4}(1-x)\le -s\le\frac{\pi}{2}$. 
We have
\begin{equation}\label{eqgamma}
|e^{is}-(x,0)|\ge |\sin s| \ge \frac{2}{\pi}|s|\ge\frac{|s|}{2}
\end{equation}
(the second inequality follows from showing that $f(t)=\frac{\sin t}{t}$ is decreasing for $0<t\le \frac{\pi}{2}$).
Moreover, estimating $|e^{it_1}-e^{it_2}|\le |t_1-t_2|$, we get
$$|e^{is}-\gamma(x,0)|\le \left|s+\frac{\pi}{2}(1-x)\right|\le |s|\le 2  |e^{is}-(x,0)|.$$
\noindent\textbf{Case 2:} $-s\le \frac{\pi}{4}(1-x)\le \frac{\pi}{2}$. We have
\begin{equation}\label{eqgamma2}
|e^{is}-(x,0)|\ge 1-x,
\end{equation}
since the circle centred at $(x,0)$ of radius $1-x$ is contained $D^2$ and $e^{is}\in \partial D^2$. Then
$$|e^{is}-\gamma(x,0)|\le \left|s+\frac{\pi}{2}(1-x)\right|\le \frac{\pi}{2}(1-x) \le  2|e^{is}-(x,0)|.$$
Then \eqref{claimgamma0} is proven. In order to prove \eqref{claimgamma}, again by symmetry, it suffices to prove it for $0\le s\le \frac{\pi}{2}$ and we consider three cases.


\noindent\textbf{Case 1:} $0\le x\le 1$, $ \frac{\pi}{2}(1-x)\le s\le \frac{\pi}{2}$.
As before we have \eqref{eqgamma}.
Moreover $|e^{is}-\gamma(x,0)|\le s+\frac{\pi}{2}(1-x)$.  Hence, by our assumption,
$$\frac{|e^{is}-\gamma(x,0)|}{4}\le \frac{s}{4}+ \frac{\pi}{8}(1-x)\le \frac{s}{4}+\frac{s}{4}\le |e^{is}-(x,0)|.$$

\noindent\textbf{Case 2:} $0\le x\le 1$, $0\le s\le \frac{\pi}{2}(1-x)\le \frac{\pi}{2}$. We have \eqref{eqgamma2} as before. Then
$$|e^{is}-\gamma(x,0)|\le s+\frac{\pi}{2}(1-x)\le \pi (1-x)\le 4(1-x)\le 4 |e^{is}-(x,0)|.$$

\noindent\textbf{Case 3:} $-1\le  x< 0$, $0\le s\le \frac{\pi}{2}$. Since
$|e^{is}-(x,0)|\ge 1$, we have
$$|e^{is}-\gamma(x,0)|\le 2 \le  2 |e^{is}-(x,0)|. $$
Then also \eqref{claimgamma} is proven.

\medskip

Now \eqref{claimgamma0} together with \eqref{diff} gives
\[\begin{split}
\int_{[-1,1]\times \{0\}}\int_{\sph^1_-} \frac{|u^*_-(e^{is})-u^*_-(x,0)|^2}{|e^{is}-(x,0)|^2}ds dx&\le {4} \int_{[-1,1]\times \{0\}}\int_{\sph^1_-} \frac{|u(e^{is})-u(\gamma(x,0))|^2}{|e^{is}-\gamma(x,0)|^2}dsdx\\
&= \frac{8}{\pi} \int_{\sph^1_-}\int_{\sph^1_-}\frac{|u(e^{is})-u(e^{it})|^2}{|e^{is}-e^{it}|^2}dsdt.
\end{split}\]
With \eqref{stimau*}, it follows that
$$[u^*_-]_{H^{\frac12,2}(\partial D^2_-)}\le C[u]_{H^{\frac12,2}(\sph^1_-)}.$$
For \eqref{claimgamma} we infer
\[\begin{split}
\int_{[-1,1]\times \{0\}}\int_{\sph^1_+} \frac{|u^*_+(e^{is})-u^*_+(x,0)|^2}{|e^{is}-(x,0)|^2}dsdx& \le 16\int_{[-1,1]\times \{0\}}\int_{\sph^1_+} \frac{|u(e^{is})-u(\gamma(x,0))|^2}{|e^{is}-\gamma(x,0)|^2}dsdx\\
&=\frac{32}{\pi}  \int_{\sph^1_-}\int_{\sph^1_+}\frac{|u(e^{is})-u(e^{it})|^2}{|e^{is}-e^{it}|^2}dsdt.
\end{split}\]
Similarly
$$\int_{[-1,1]\times \{0\}}\int_{\sph^1_+} \frac{|v^*(e^{is})-v^*(x,0)|^2}{|e^{is}-(x,0)|^2}dsdx\le \frac{32}{\pi} \int_{\sph^1_-}\int_{\sph^1_+}\frac{|v(e^{is})-u(e^{it})|^2}{|e^{is}-e^{it}|^2}dsdt.$$
Together with \eqref{stimau*}, it follows that 
$$[u^*_+]_{H^{\frac12,2}(\partial D^2_+)}\le C[u]_{H^{\frac12,2}(\sph^1)},\qquad [v^*]_{H^{\frac12,2}(\partial D^2_+)}\le C[v\&u]_{H^{\frac12,2}(\sph^1)}.$$
Let us now call $U_+\in H^{1,2}(D^2_+,\mathbb{C})$, $U_-\in H^{1,2}(D^2_-,\mathbb{C})$,  $V\in H^{1,2}(D^2_+,\mathbb{C})$ any extensions of $u^*_+$, $u^*_-$ and $v^*$, respectively, which exist thanks to Proposition \ref{trace}.
Since $U_+$, $U_-$ and $V$ have the same trace $u^*$ on $(-1,1)\times \{0\}$ they can be patched together to obtain extensions $U$, $V\&U$ and $V\#U$ of $u$, $v\&u$ and  $v\#u$, respectively, namely
\[
U=\begin{cases}
U_+&\text{in }D^2_+\\
U_-&\text{in }D^2_-,
\end{cases}
\qquad
V\&U=\begin{cases}
V&\text{in }D^2_+\\
U_-&\text{in }D^2_-,
\end{cases}
\]
and
\[
V\#U(z)=\begin{cases}
V(z)&\text{for }z\in D^2_+\\
U_+(\bar z)&\text{for }z\in D^2_-
\end{cases}
\]
Then, thanks to \eqref{degformula2}
\[\begin{split}
\deg(u\&v)&=\frac{1}{\pi}\int_{D^2} J(V\&U)\, dxdy\\
&=\frac{1}{\pi}\left(\int_{D^2_+} JV\, dxdy-\int_{D^2_+} JU_+\, dxdy\right)+ \frac{1}{\pi}\left(\int_{D^2_+} JU_+\, dxdy+\int_{D^2_-} JU_-\, dxdy\right)\\
&=\frac{1}{\pi}\int_{D^2} J(V\#U)\, dxdy+ \frac{1}{\pi}\int_{D^2} JU\, dxdy\\
&=\deg(v\#u)+\deg(u).
\end{split}\]
\end{proof}

Composing with a suitable conformal map from $D^2$ to $\R^2_+$
Proposition \ref{p:degjump} implies:

\begin{prop}\label{p:degjump2} Consider $u,v \in \dot H^{\frac{1}{2},2}(\R,\sph^1)$ and $\ve>0$ such that for
\[v\&u(x):=
\begin{cases}
v(x) &\text{for }|x|\le \ve \\
u(x) &\text{for }|x|>\ve
\end{cases}\]
and
\[v\# u(x):=
\begin{cases}
v(x) &\text{for }|x|\le \ve\\
u\big(\frac{\ve^2}{x}\big) &\text{for }|x|>\ve,
\end{cases}\]
one has $v\&u, v\#u\in \dot H^{\frac{1}{2},2}(\R,\sph^1)$. Then
\begin{equation}\label{eq:jump}
\deg(v\&u)= \deg(u)+\deg (v\#u).
\end{equation}
\end{prop}

\section{Proof of Theorems \ref{mainthm} and \ref{mainthm2}}

Since Theorem \ref{mainthm} follows immediately from Theorem \ref{mainthm2}, we will only need to prove the latter.

The following lemma is the crucial energy estimate, and is the analog of Lemma 2 in \cite{bc}.

\begin{lem}\label{strictestimate} Let $u\in \dot H^{\frac{1}{2},2}(\R,\sph^1)$ be such that for some $x_0\in \R$, $u$ is smooth near $x_0$, $u'(x_0)\ne 0$ and 
\begin{equation}\label{12harm}
\partial_y \tilde u(x_0,0)\perp T_{u(x_0)}\sph^1,
\end{equation}
where $\tilde u$ is the Poisson harmonic extension of $u$.
Then for every $\ve>0$ sufficiently small there exists $v \in \dot H^{\frac{1}{2},2}(\R,\sph^1)$ such that $v=u$ in $\R\setminus [x_0-2\ve,x_0+2\ve]$, $\deg v= \deg u\pm1$ and
$$E(v)< E(u)+2\pi.$$ More precisely,
$\deg v= \deg u- \mathrm{sign}(u(x_0)\wedge u'(x_0))$ and we can choose $v$ (depending on $\ve>0$) such that
$$E(v)\le E(u)+2\pi -(1-\ln 2)|u'(x_0)|^2\ve^2+o(\ve^2),\quad \text{as }\ve\to 0.$$
\end{lem}

\begin{proof} Using the harmonic extension, we are able to closely follow \cite[pp. 206-207]{bc}. Up to a translation and a rotation we can set $x_0=0$, $u(x_0)=-i=(0,-1)$. These assumptions and \eqref{12harm} give
\begin{equation}\label{nablatildeu}
\nabla\tilde u(0,0)=\begin{pmatrix}
a&0\\
0&b
\end{pmatrix}
\end{equation}
where $u'(0)=(a,0)\in T_{(0,-1)}\sph^1=\R\times\{0\}$, $\partial_y\tilde u(0,0)=(0,b)\perp T_{(0,-1)}\sph^1$. Up to a reflection, we can also assume $a>0$.

We now consider the inverse stereographic projection from the North pole
$$\Pi^{-1}_+:\R\to \sph^1,\quad \Pi^{-1}_+(x)=\left( \frac{2x}{1+x^2}, \frac{-1 +x^2}{1+x^2}\right),$$ which 
can be extended to the biholomorphic map
$\Phi:\cl{\R^{2}_+}\to \cl{D^2}\setminus\{i\},$
which, using the complex notation $z=x+iy$, is
\begin{equation}\label{Phi}
\Phi(z)=\frac{z-i}{1-iz}=\left( \frac{2x}{(1+y)^2+x^2}, \frac{-1+x^2+y^2}{(1+y)^2+x^2}\right).
\end{equation}
Notice that $\Phi(\cl{\R^2_+})= \cl{D^2}\setminus\{i\}$.
Its inverse is
\begin{equation}\label{Psi}
\Psi:\cl{D^2} \setminus\{i\}\to \cl{\R^{2}_+},\quad \Psi(z)=\Phi^{-1}(z)=-i\frac{z+i}{z-i}.
\end{equation}
Similarly, the inverse stereographic projection from the South pole
$$\Pi_-^{-1}(x)=\left( \frac{2x}{1+x^2}, \frac{1-x^2}{1+x^2}\right)$$
can be extended to the anti-biholomorphic map $\bar\Phi:\cl{\R^{2}_+}\to \cl{D^2}\setminus\{-i\}$,
$$
\bar \Phi(z)=\frac{\bar z+i}{1+i\bar z}=\left( \frac{2x}{(1+y)^2+x^2}, \frac{1-x^2-y^2}{(1+y)^2+x^2}\right).
$$
We now set for $\lambda=\mu\ve^2$, with $\ve>0$ sufficiently small and $\mu>0$ to be fixed depending on $a$ and $b$ only, 
\begin{equation*}
U_\ve(z)=
\begin{cases}
\bar \Phi_\lambda(z):=\bar \Phi(\tfrac{z}{\lambda})=\frac{\frac{\bar z}{\lambda}+i}{1+i\frac{\bar z}{\lambda}} & \text{in }B_\ve^+(0)\\
\Phi\begin{pmatrix}
A_1r+B_1\\
A_2r+B_2
\end{pmatrix}&\text{in } B_{2\ve}^+(0)\setminus B_{\ve}^+(0)\\
\tilde u(z)&\text{in }\R^2_+\setminus B_{2\ve}^+(0)
\end{cases}
\end{equation*}
with $A_1,A_2,B_1,B_2$ to be determined as functions of $\ve,\theta$ to make $U_\ve$ continuous. In particular, since $\Psi=\Phi^{-1}$, we want
\begin{equation}\label{eq2ve}
\begin{pmatrix}
A_1(\theta)r+B_1(\theta)\\
A_2(\theta)r+B_2(\theta)
\end{pmatrix} =\Psi(\tilde u(re^{i\theta})) \quad \text{for }r=2\ve,
\end{equation}
and
\begin{equation}\label{eqve}
\begin{pmatrix}
A_1(\theta)r+B_1(\theta)\\
A_2(\theta)r+B_2(\theta)
\end{pmatrix} =\Psi(\bar \Phi_\lambda(re^{i\theta}))\quad \text{for }r=\ve.
\end{equation}
Notice that this implies that $A_2(0)=A_2(\pi)=B_2(0)=B_2(\pi)=0$, since $\Psi\circ \tilde u$ and $\Psi\circ \bar \Phi_\lambda$ send $\R\times\{0\}$ into itself. It follows in particular that $U_\ve$ sends $\R\times \{0\}$ into $\sph^1$.

Using polar coordinates at $(0,0)$, from \eqref{nablatildeu} we get (identifying $i$ with $(0,-1)$, with a little abuse of notation) the Taylor expansion 
\begin{align*}
\tilde u(re^{i\theta})&=-i+\begin{pmatrix}a&0\\0&b\end{pmatrix}\begin{pmatrix}r\cos\theta\\ r\sin\theta\end{pmatrix}+O(r^2)=-i+\begin{pmatrix}ar\cos\theta\\br\sin\theta\end{pmatrix}+O(r^2).
\end{align*}
We then compute from \eqref{Psi}
$$\Psi(z)=\frac{1}{2}(z-(-i))+O((z-(-i))^2),\quad \text{as }z\to -i, $$
so that
$$\Psi(\tilde u(re^{i\theta}))=\frac{1}{2}\begin{pmatrix}ar\cos\theta\\br\sin\theta\end{pmatrix}+O(r^2)\quad \text{as }r\to 0.$$
This and \eqref{eq2ve} give
\begin{equation}\label{sys1}
\begin{cases}
2\ve A_1(\theta)+B_1(\theta)=a\ve\cos\theta+O(\ve^2)\\
2\ve A_2(\theta)+B_2(\theta)=b\ve\sin\theta+O(\ve^2).
\end{cases}
\end{equation}
Moreover
\begin{equation}\label{barPhi0}
\bar \Phi_\lambda(z)=-i+2\frac{\lambda}{\bar z}+O(\ve^2),\quad \text{for }|z|=\ve,
\end{equation}
hence 
$$\Psi(\bar  \Phi_\lambda(\ve e^{i\theta}))= \frac{\lambda}{\ve e^{-i\theta}}+O(\ve^2) =\mu\ve\begin{pmatrix}\cos\theta\\\sin\theta\end{pmatrix}+O(\ve^2),$$
and together with \eqref{eqve} we get
\begin{equation}\label{sys2}
\begin{cases}
\ve A_1(\theta)+B_1(\theta)=\mu\ve\cos\theta+O(\ve^2)\\
\ve A_2(\theta)+B_2(\theta)=\mu\ve\sin\theta+O(\ve^2).
\end{cases}
\end{equation}
Solving \eqref{sys1}-\eqref{sys2} yields
\begin{equation}\label{AB}
\begin{cases}
A_1(\theta)=(a-\mu)\cos\theta+O(\ve)\\
A_2(\theta)=(b-\mu)\sin\theta+O(\ve)\\
B_1(\theta)=\ve(2\mu-a)\cos\theta+O(\ve^2)\\
B_2(\theta)=\ve(2\mu-b)\sin\theta+O(\ve^2).
\end{cases}
\end{equation}
We now easily estimate
\begin{equation}\label{est1}
\begin{split}
\int_{\R^2_+\setminus B_{2\ve}^+(0)}|\nabla U_\ve|^2dxdy&= E(u)-\int_{B_{2\ve}^+(0)}|\nabla \tilde u|^2dxdy=E(u)-2\pi\ve^2(a^2+b^2)+o(\ve^2).
\end{split}
\end{equation}
Observing that $\bar\Phi_\lambda$ is anticonformal, so that $|\nabla\bar\Phi_\lambda|^2=2|J\bar\Phi_\lambda|$, considering that $\bar\Phi_\lambda$ is a diffeomorphism of $\R^2_+$ onto $D^2$, and taking \eqref{barPhi0} into account, we see that $\bar\Phi_\lambda(B_\ve^+(0))$ is the unit disk minus a region of area $2\pi\mu^2\ve^2+o(\ve^2)$, so that
\begin{equation}\label{est2}
\begin{split}
\int_{B_\ve^+(0)}|\nabla\bar\Phi_\lambda|^2dxdy&=2\mathrm{Area}\(\bar\Phi_\lambda(B_\ve^+(0))\) =2\pi-4\pi\mu^2\ve^2+o(\ve^2).
\end{split}
\end{equation}
From \eqref{Phi} we obtain $\nabla\Phi(z)= 2Id+o(1)$ as $z\to 0$, and by the chain rule
we finally obtain 
\begin{equation}\label{est3}
\begin{split}
\int_{B_{2\ve}^+(0)\setminus B_\ve^+(0)} |\nabla U_\ve|^2dxdy&=4(1+o(1))\int_{B_{2\ve}^+(0)\setminus B_\ve^+(0)}\left|\nabla \begin{pmatrix}
A_1(\theta)r+B_1(\theta)\\
A_2(\theta)r+B_2(\theta)
\end{pmatrix}\right|^2dxdy\\
&=4(1+o(1))\int_0^{\pi}\int_\ve^{2\ve}\sum_{\ell=1}^2\left[ A_\ell^2+\( A_\ell' +\frac{B_\ell'}{r}\)^2\right]rdrd\theta\\
&=2\pi\ve^2[a^2+b^2-2\mu^2+(a^2+b^2+8\mu^2-4a\mu-4b\mu)\ln 2]+o(\ve^2).
\end{split}
\end{equation}
Summing \eqref{est1}, \eqref{est2} and \eqref{est3} we get 
$$\int_{\R^2_+}|\nabla U_\ve|^2dxdy=E(u)+2\pi-2\pi\ve^2(4\mu^2-(8\mu^2+a^2+b^2-4a\mu-4b\mu)\ln2)+o(\ve^2). $$
Then, as in \cite{bc}, we can choose $\mu=\max\{a/2,b/2\}$ to get  
\[\begin{split}
4\mu^2-(8\mu^2+a^2+b^2-4a\mu-4b\mu)\ln2&=\max\{a^2,b^2\}-(a-b)^2\ln2\\
&\ge (1-\ln 2)\max\{a^2,b^2\}\ge (1-\ln 2)a^2>0.
\end{split}\]
Now, setting $v=u_\ve=U_\ve|_{\R\times\{0\}}$ and recalling \eqref{equiv1},
we conclude
$$E(v)\le  \int_{\R^2_+}|\nabla U_\ve|^2dxdy\le E(u)+2\pi-(1-\ln 2)a^2\ve^2+o(\ve^2)< E(u)+2\pi,$$
for some $\ve>0$ small enough. The finiteness of $E(v)$ implies that $v\in \dot H^{\frac12,2}(\R,\mathbb{C})$ and we have already noticed that $U_\ve|_{\R\times\{0\}}$ takes values into $\sph^1$, so that $v\in \dot H^{\frac12,2}(\R,\sph^1)$.

Let us now set
\begin{align*}
w\& u(x):=\left\{ \begin{array}{ll}w(x)&\quad\text{for }|x|\leq2\ve \\ u(x)&\quad\text{for }|x|\geq2\ve,  \end{array} \right. \quad w\# u(x):=\left\{ \begin{array}{ll}w(x)&\quad\text{for }|x|\leq2\ve\\ u(\frac{4\ve^2}{x})&\quad\text{for }|x|\geq2\ve,  \end{array} \right.    \end{align*} so that $v=w\&u$, where  $w(x):=U_\ve(x,0)$ for $|x|\leq 2\ve.$ 
Notice that $w\#u \in \dot H^{\frac12,2}(\R,\sph^1)$, since 
\begin{align*}
[w\#u]_{H^{\frac12,2}(\R)}^2&= \int_\R\int_\R\frac{|w\#u(x)-w\#u(y)|^2}{|x-y|^2}dxdy
\\
&\leq 2\left( \int_{I_{2\delta}}\int_{I_{2\delta}} +\int_{\R\setminus I_{\delta}}\int_{ \R\setminus I_\delta}  +\int_{\R\setminus I_{2\delta}}\int_{ I_\delta}     \right)\frac{|w\#u(x)-w\#u(y)|^2}{|x-y|^2}dxdy , \end{align*}
where $I_\delta:=(-\delta,\delta)$, $\delta>0$ is fixed such that $u$ is smooth in $I_{3\delta}$, and $\ve\in (0,\delta/2)$. Then $w\#u$ is Lipschitz continuous on $I_{2\delta}$, and in particular the first integral on the right-hand side above is finite. The finiteness of the second integral follows by changing the variable $x\mapsto \frac1x,\,y\mapsto\frac1y$, using that $w\#u(x)=u(\frac{4\ve^2}{x})$ for $x\in \R\setminus I_\delta$ and $u\in \dot H^{\frac12,2}(\R,\sph^1)$. For the convergence of the third integral, one can use $|x-y|\geq \frac{|y|}{2}$ for $(x,y)\in I_\delta\times(\R\setminus I_{2\delta})$, while the enumerator is bounded.  

Since $u$ is smooth near the origin, it follows from the estimates on $A_1$ and $B_1$ that $\sph^1_+$  is covered by the   function  $w\#u $ only once in clockwise direction,  and hence  $\deg(w\#u)=-1$. Therefore, by Proposition \ref{p:degjump2} $\deg(v)=\deg(u)-1.$
\end{proof}

\begin{cor}\label{c:est} Let $u_0$ minimize $E$ in $\M{E}_g$, with $g\in \dot H^{\frac12,2}(\R,\sph^1)$ not constant in $\R\setminus I$. Then there exists $x_0\in I$ such that $u'(x_0)\ne 0$, and setting $k=\deg(u_0)-\mathrm{sign}(u(x_0)\wedge u'(x_0))$, we have
\begin{equation}\label{eqc:est}
\inf_{\M{E}_{g,k}} E< \inf_{\M{E}_g}E+2\pi.
\end{equation}
\end{cor}

\begin{proof} Since $u_0$ is half-harmonic, it is smooth in $I$ by \cite{DLR}. By Lemma \ref{l:unicont} below, $u_0$ in not constant in $I$, hence there exists $x_0\in I$ such that $u'(x_0)\ne 0$. By \cite[Sec. 1]{milsir} $u_0$ satisfies condition \eqref{12harm}, hence Lemma \ref{strictestimate} applies.
\end{proof}

\noindent\emph{Proof of Theorem \ref{mainthm2} (completed).} Given $u_0$ minimizing $E$ in $\M{E}_g$, by Corollary \eqref{c:est} we can take $x_0\in I$ such that $u_0'(x_0)\ne 0$ and set $k=\deg(u_0)-\mathrm{sign}(u(x_0)\wedge u'(x_0))$. Consider a minimizing sequence $(u_n)\subset \M{E}_{g,k}$. Up to a subsequence $u_n\rightharpoonup  u^*$ weakly in $\dot H^{\frac12,2}(\R,\sph^1)$ where $ u^* \in  \M{E}_{g,\ell}$ for some $\ell\in\mathbb{Z}$. According to Proposition \ref{p:jump}
$$\inf_{\M{E}_g}E\le E(u^*)\le \liminf_{n\to\infty} E(u_n)-2\pi |k-\ell|= \inf_{\M{E}_{g,k}}E -2\pi |k-\ell|.$$
If $k\ne\ell$ this contradicts \eqref{eqc:est} of Corollary \ref{c:est}, hence $\ell=k$ and $ u^*\in \M{E}_{g,k}$ is the desired minimizer of $E$ in $\M{E}_{g,k}$.\hfill$\square$



\begin{lem}\label{l:unicont} Let $u\in \dot H^{\frac12,2}(\R,\sph^1)$ be both half-harmonic and constant in $I=(-1,1)$. Then $u$ is constant in all of  $\R$.
\end{lem}

\begin{proof} Up to a rotation, we may assume that $u\equiv i=(0,1)$ in $I$. Writing $u=v+iw$, with $v$ and $w$ real valued, we obtain from \eqref{eq12harm} that
\begin{equation}\label{v12harm}
v=(-\Delta)^\frac12 v=0\quad \text{in }I.
\end{equation}
This implies that $v\equiv 0$ everywhere in $\R$ by the unique continuation property of fractional Laplacians, see e.g. \cite[Thm. 1.2]{GSU}. An elementary proof of this fact, in this simple situation, can also be obtained by considering the Poisson harmonic extension $\tilde v$ of $v$ (which is smooth up to $I\times \{0\}$ by elliptic regularity).
Then by \eqref{v12harm}
$$\frac{\de \tilde v(x,y)}{\de y}\bigg|_{y=0}=0,\quad \text{for }x\in I.$$
Considering $B_-:=\{(x,y)\in \R^2: x^2+y^2<1,\,y<0\}$ and
\[
\tilde{\tilde v}=\begin{cases}
\tilde v & \text{in }\R^2_+\\
0& \text{in }B_{-},
\end{cases}
\]
we obtain that $\tilde{\tilde v}$ is harmonic, hence it vanishes identically in $\R^2_+\cup B_-$ by the unique continuation property of harmonic maps. Then $v\equiv 0$ in $\R$.

Since $v^2+w^2=1$ a.e., it follows that $w(x)\in \{-1,1\}$ for a.e.   $x$, and since $w\in \dot H^{\frac12,2}(\R,\R)$, $w$ must be constant, hence $w\equiv 1$ and $u\equiv (0,1)$,  see \cite[Corollary 6.2]{BM}.
\end{proof}

\section{Proof of Theorem \ref{thmex}}

Let us first recall a few facts about Blaschke products.
Blaschke products are holomorphic maps $H:\cl{D^2}\subset \mathbb{C}\to \cl{D^2} \subset \mathbb{C}$ of the form
$$H(z)= e^{i\theta_0}\prod_{j=1}^k \frac{z-a_j}{1-\bar a_j z},\quad a_1,\dots, a_k\in D^2,\, \theta_0\in \R.$$
They are the only holomorphic maps of $\cl{D^2}$ into itself sending $\sph^1$ into itself.
Then, by conformality, as in the degree formula \eqref{degformula2}
$$\int_D |\nabla H|^2 dx dy = 2 \int_{D} JH dxdy =2\pi \deg H|_{\sph^1}=2\pi k.$$
By invariance of the Dirichlet energy under conformal trasformations, we also have for    $h(x):=H\circ \Pi_+^{-1}$ and its Poisson extension $\tilde h =H\circ\Phi$
$$E(h)=\int_{\R^2_+}|\nabla \tilde h|^2 dxdy=\int_{D^2}|\nabla H|^2dxdy=2\pi k.$$
Every such map has the form
\begin{equation}\label{Blas}
\tilde h(z)= e^{i\theta_0}\prod_{j=1}^k\frac{z-\beta_j}{z-\bar \beta_j},\quad \beta_1,\dots\beta_k\in \R^2_+,\, \theta_0\in \R.
\end{equation}
We define for $k>0$ the set of Blaschke products
$$\M{B}_k=\left\{\tilde h\text{ as in }\eqref{Blas}\right\} $$
and the set of their conjugates
$$\M{B}_{-k}=\left\{\bar{\tilde h}: \tilde h \in \M{B}_k\right\}.$$
For $k=0$, $\M{B}_0$ are just constant maps into $\sph^1$.

\medskip

Notice that for $\tilde h\in \M{B}_{\pm k}$ and $h(x)=\tilde h(x,0)$, we have $\deg(h)=\pm k$, $E(h)=2\pi|k|$.

\medskip

The following lemma is well known, see e.g. \cite[Lemma 3.1]{BMRS} or \cite[Thm. 12.10]{BM}. We state it and sketch its proof for completeness.

\begin{lem}\label{l:Blas} Consider for $k\in\mathbb{Z}$
$$\M{E}_k:=\{v\in \dot H^{\frac12,2}(\R,\sph^1):\deg(v)=k\}.$$
Then
$$\inf_{\M{E}_k}E=2\pi|k|$$
and the infimum is attained only by the restrictions of  maps $\tilde h\in \M{B}_k$.  
\end{lem}

\begin{proof} (Sketch) By conformal invariance of $E$, we can prove the lemma for functions in $H^{\frac12,2}(\sph^1,\sph^1)$. By the degree formula \eqref{degformula2} we have for $v\in \M{E}_k$
$$2\pi |k|=2\left|\int_D J\tilde v dxdy\right|\le \int_D|\nabla \tilde v|^2dxdy=: E(v), $$
where $\tilde v$ is the harmonic extension of $v$ to the whole disk, and the inequality is strict unless $\tilde v$ is everywhere conformal or anticonformal, which corresponds to $v\in \M{B}_k$ or $v\in \M{B}_{-k}$, respectively.
\end{proof}

\noindent\emph{Proof of Theorem \ref{thmex} - Part (i).} Since $g_\lambda$ is the restriction to $\R\times\{0\}$ of $\Phi_\lambda:=\Phi\(\tfrac{\cdot}{\lambda}\)\in\M{B}_1$, the minimality of $g_\lambda$ and its uniqueness follow at once from Lemma \ref{l:Blas}.

\medskip
\noindent\emph{Part (ii).}
We have $E(u_0)=E(g_\lambda)=2\pi$.  Except when $k=1$, none of the Blaschke products in $\M{B}_k$ matches $g_\lambda$ on $\R\setminus I$, hence, by Lemma \ref{l:Blas}, for every $u\in \M{E}_{g_\lambda,k}$, $k\in \mathbb{Z}\setminus \{0,1\}$
$$E(u)>2\pi|k|\ge 2\pi = E(u_0).$$
Still by Lemma \ref{l:Blas}, $E(u)>2\pi =E(u_0)$ for every $u\in \M{E}_{g_\lambda,1}\setminus\{u_0\}$.
It remains to exclude for $\lambda$ sufficiently large the existence of a function $u\in \M{E}_{g_\lambda,0}$ such that $E(u)\le E(g_\lambda)$.

Consider $\lambda_n\to\infty$. We have by dominated convergence
\[
\begin{split}
\lim_{n\to\infty}\frac{1}{2\pi}\int_{\R\setminus I}\int_{\R\setminus I}\frac{|g_{\lambda_n}(x)-g_{\lambda_n}(\xi)|^2}{(x-\xi)^2}dxd\xi&=\lim_{n\to\infty}\frac{1}{2\pi}\int_{\R\setminus \left[-\tfrac{1}{\lambda_n},\tfrac{1}{\lambda_n}\right]}\int_{\R\setminus \left[-\tfrac{1}{\lambda_n},\tfrac{1}{\lambda_n}\right]}\frac{|g_1(x)-g_1(\xi)|^2}{(x-\xi)^2}dxd\xi\\
&=\frac{1}{2\pi}[g_1]_{H^{\frac{1}{2},2}(\R)}^2=E(g_1)=2\pi,
\end{split}
\]
while, for any sequence of functions $(u_n)\subset \M{E}_{g_{\lambda_n},0}$ we have
\begin{equation}\label{unlarge}
\liminf_{n\to\infty}\frac{1}{2\pi}\left(\int_{\R}\int_{I}\frac{|u_n(x)-u_n(\xi)|^2}{(x-\xi)^2}dxd\xi+\int_{I}\int_{\R\setminus I}\frac{|u_n(x)-u_n(\xi)|^2}{(x-\xi)^2}dxd\xi \right)  \ge 2\pi.
\end{equation}
In order to prove \eqref{unlarge}, consider the function
\begin{equation}\label{defvn}
v_n(x)=\begin{cases}
u_n(x)& \text{for }|x|\le 1\\[3mm]
g_{\lambda_n}\(x\)&\text{for }1<|x|\le R\\[3mm]
g_{\lambda_n}\(\tfrac{R^2}{ x}\) &\text{for }|x|>R.
\end{cases}
\end{equation}
For $R\geq2$ we rewrite the functions $u_n$ and $v_n$ as $$u_n=w_1\& w_2,\quad v_n=w_1\#w_2,\quad\text{where }\quad w_1:=u_n|_{[-R,R]},\quad  w_2:=g_{\lambda_n}, $$ and hence, by Proposition \ref{p:degjump2} $$\deg(v_n)=\deg(u_n)-\deg(w_2)=-1.$$
By Lemma \ref{l:Blas} we have
$$E(v_n)\ge 2\pi.$$
On the other hand, from \eqref{defvn} we estimate 
\begin{equation}\label{vno1}\int_{\R\setminus I}\int_{\R\setminus I}\frac{|v_n(x)-v_n(\xi)|^2}{(x-\xi)^2}dxd\xi=o(1),\quad \text{as }n\to\infty.
\end{equation} Indeed, for $R\geq 2$ writing \begin{align*}  & \int_{\R\setminus I}\int_{\R\setminus I}\frac{|v_n(x)-v_n(\xi)|^2}{(x-\xi)^2}dxd\xi\\&=\left(\int_{\R\setminus [-R,R]}\int_{\R\setminus [-R,R]}+2\int_{\R\setminus [-R,R]}\int_{[-R,R]\setminus I}+\int_{[-R,R]\setminus I}\int_{[-R,R]\setminus I}  \right)\frac{|v_n(x)-v_n(\xi)|^2}{(x-\xi)^2}dxd\xi\\&=:(I)+2(II)+(III),\end{align*} one can bound with a change of variables in $x$ and in $\xi$
\begin{align*}(I)&=\int_{[-R,R]}\int_{[-R,R]} \frac{|g_{\lambda_n}(x)-g_{\lambda_n}(\xi)|^2}{(x-\xi)^2}dxd\xi\\
&\leq \|g_1'\|_{L^\infty}\frac{1}{\lambda_n^2} \int_{[-R,R]}\int_{[-R,R]}dxd\xi\leq\frac{C(R)}{\lambda_n^2}, \end{align*}
with a change of variables in $\xi$  \begin{align*} (II)&= \int_{[-R,R]}\int_{[-R,R]\setminus I}\frac{|g_{\lambda_n}(x)-g_{\lambda_n}(\xi)|^2}{(x-\frac{R^2}{\xi})^2}\frac{R^2}{\xi^2}dxd\xi \\&\leq R^2 \frac{\|g_1'\|}{\lambda_n^2}\int_{[-R,R]}\int_{[-R,R]\setminus I}\frac{( x-\xi)^2}{(x\xi- R^2)^2} dxd\xi \leq \frac{C(R)}{\lambda_n^2}, \end{align*}
and finally
\begin{align*} (III)=\int_{[-R,R]\setminus I}\int_{[-R,R]\setminus I} \frac{|g_{\lambda_n}(x)-g_{\lambda_n}(\xi)|^2}{(x-\xi)^2}dxd\xi\leq (I)\leq  \frac{C(R)}{\lambda_n^2}, \end{align*}
so that $\eqref{vno1}$ is proven.

Moreover, since $R\ge 2$, we have $|x-y|\ge \frac{|y|}{2}$ for $x\in I$, $|y|\ge R$, so that
\[\begin{split}
\int_{\R\setminus[-R,R]}\int_{I}\frac{|v_n(x)-v_n(\xi)|^2}{(x-\xi)^2}dxd\xi&\le
\int_{\R\setminus[-R,R]}\int_{I}\frac{16}{\xi^2}dxd\xi= \frac{64}{R}.
\end{split}\]
It then follows
\begin{align*}2\pi \le E(v_n)\le & \frac{1}{2\pi}\int_{[-R,R]}\int_{I}\frac{|v_n(x)-v_n(\xi)|^2}{(x-\xi)^2}dxd\xi+\frac{1}{2\pi}\int_{I}\int_{[-R,R]\setminus I}\frac{|v_n(x)-v_n(\xi)|^2}{(x-\xi)^2}dxdy\\&\quad +\frac{C}{R}+o(1),\quad \text{as }n\to\infty\end{align*}
and, since $u_n=v_n$ in $[-R,R]$,
\[
\begin{split}
\frac{1}{2\pi}\int_{[-R,R]}\int_I\frac{|u_n(x)-u_n(\xi)|^2}{(x-\xi)^2}dxdy+\frac{1}{2\pi}\int_{I}\int_{[-R,R]\setminus I}\frac{|u_n(x)-u_n(\xi)|^2}{(x-\xi)^2}dxd\xi\ge 2\pi-\frac{C}{R}+o(1)
\end{split}
\]
as $n\to\infty$. Since $R>1$ can be taken arbitrarily large, \eqref{unlarge} is proven.


Since 
\begin{align*}E(u)&=\frac{1}{2\pi}\int_{\R\setminus I}\int_{\R\setminus I}\frac{|u(x)-u(\xi)|^2}{(x-\xi)^2}dxd\xi+\frac{1}{2\pi}\int_{\R}\int_{I}\frac{|u(x)-u(\xi)|^2}{(x-\xi)^2}dxd\xi\\&\qquad +\frac{1}{2\pi}\int_{I}\int_{\R\setminus I}\frac{|u(x)-u(\xi)|^2}{(x-\xi)^2}dxd\xi,\end{align*}
we conclude that
$$\liminf_{n\to\infty}E(u_n) \ge 4\pi,$$
hence for $\lambda$ sufficiently large we have $E(u)>2\pi$ for every $u\in \M{E}_{g_\lambda,0}$.

\medskip

\noindent\emph{Part (iii).} Let $u_0$ be an absolute minimizer of $E$ in $\M{E}_{g_\lambda}$, and set $k=\deg(u_0)$. If $k=0$ then we are done. For $|k|\geq1 $,  it follows from  Lemma \ref{l:Blas} that  $$2\pi\leq2\pi|k|=\inf_{\M{E}_k} E\leq E(u_0)\leq E(g_\lambda)=2\pi. $$ Thus, we can simply take $u_0=g_\lambda$,  and  the existence of a minimizer $u^*$ in $\M{E}_{g_\lambda,0}$ follows at once from Theorem \ref{mainthm2}. 

On the other hand, by Proposition \ref{propex} below, there exists $\lambda_0>1$ such that for $\lambda\in (0,\lambda_0)$ there exists $u_0\in \M{E}_{g_\lambda,0}$ absolute mimimizer of $E$ in $\M{E}_{g_\lambda}$ with $E(u_0)<2\pi=E(g_\lambda)$.

\medskip

\noindent\emph{Part (iv).} We claim that for any $k>1$ and $\lambda>0$
\begin{equation}\label{est2pik}
\inf_{\M{E}_{g_\lambda,k}}E\le E(g_\lambda)+ 2\pi|k-\deg(g_\lambda)|= 2\pi k.
\end{equation}
In order to prove \eqref{est2pik}, take $\delta>0$ sufficiently small, $k-1$ distinct points $x_1,\dots, x_{k-1}\in I$ and $v\in \dot H^{\frac{1}{2},2}(\R,\sph^1)$ such that 
\begin{enumerate}
\item $|x_j|<1-\delta$ for $1\le j\le k-1$ and $|x_j-x_\ell|>2\delta$ for $1\le j<\ell\le k-1$,
\item $v=g_\lambda$ in $\R\setminus \cup_{j=1}^{k-1} (x_j-\delta,x_j+\delta)$,
\item $v(x_j)\wedge v'(x_j)<0$,
\item $E(v)\le E(g_\lambda)+\delta.$
\end{enumerate}
To prove that such $v$ exists, consider the map $\Pi_+\circ g_\lambda(x)=\frac{x}{\lambda}$ for $x\in \R$. For a function $\eta\in C^\infty_c((-1,1))$ with $\eta'(0)=-\frac{2}{\lambda}$ and $\rho>0$ set $\eta_\rho(x):=\rho\eta(x/\rho)$ and define
$$v_\rho:= \Pi_+^{-1}\(\Pi_+\circ g_\lambda +\sum_{j=1}^{k-1} \eta_\rho(x-x_j)\).$$
Since for $\rho$ small it holds 
$$(\Pi_+\circ v_\rho)'(x_j)=\frac{1}{\lambda}-\eta'(0)=-\frac{1}{\lambda}<0,$$
we have $v_\rho(x_j)\wedge v'_\rho(x_j)<0$ for $j=1,\dots,k-1$ since for $x$ in a neighborhood of $x_j$ we have that $v_\rho(x)$ moves in the clockwise direction on $\sph^1$ as $x$ increases. Moreover, by scale invariance,
$$[\eta_\rho]_{H^{\frac{1}{2},2}(\R)}=\rho [\eta]_{H^{\frac{1}{2},2}(\R)}=O(\rho), \quad  \text{as }\rho\to 0,$$ hence $E(v_\rho)\to E(g_\lambda)$ as $\rho\to 0$, and it suffices to choose $\rho>0$ small enough.

Then, applying Lemma \ref{strictestimate} to $v$ at $x_{1},\dots,x_{k-1}$ we obtain a new function $w$ with $\deg(w)=\deg(v)+k-1=k$ and
$$E(w)<E(v)+2\pi(k-1)\le E(g_\lambda)+\delta+2\pi(k-1)=2\pi k+\delta.$$
Since $\delta>0$ is arbitrary, \eqref{est2pik} follows.

On the other hand, by Lemma \ref{l:Blas}
$$\inf_{\M{E}_{g_\lambda,k}}E\ge\inf_{\M{E}_{k}}E =2\pi k$$
and can only be attained by Blaschke products in $\M{B}_k$. Since $\M{B}_k\cap \M{E}_{g_\lambda}=\emptyset$ for $k>1$, the infimum of $E$ in $\M{E}_{g_\lambda,k}$ is not attained.
\hfill$\square$

\begin{prop}\label{propex} Consider $g_\lambda(x)=\Pi_+^{-1}(x/\lambda)$. There exists $\lambda_0>1$ such that for $\lambda\in (0,\lambda_0)$
$$2\pi=E(g_\lambda)>\inf_{\M{E}_{g_\lambda}}E = \inf_{\M{E}_{g_\lambda,0}}E.$$
\end{prop}

\begin{proof} As in the proof of Theorem \ref{thmex}, Part (i), we have
$$E(g_\lambda)=2\pi =\min_{\M{E}_{g_\lambda,1}}E.$$
It suffices to prove that for $\lambda_0>1$ sufficiently close to $1$, and for $\lambda\in (0,\lambda_0]$, we have
\begin{equation}\label{eqgdeg0}
\inf_{\M{E}_{g_\lambda,0}} E<2\pi,
\end{equation}
i.e. there exists $u\in \M{E}_{g_\lambda}$ with $\deg(u)=0$ and $E(u)<2\pi$. Then, since a minimizer $u_0$ of $E$ in $\M{E}_{g_\lambda}$ exists by direct methods, we must have $u_0\in \M{E}_{g_\lambda,0}$ by \eqref{eqgdeg0} and Lemma \ref{l:Blas}.

We start with the case $\lambda=1$. It is convenient to work on $\sph^1$ via stereographic projection. Clearly $g_1\circ \Pi_+(e^{i\theta})=e^{i\theta}$. Now we define
\[
v(e^{i\theta})=
\begin{cases}
e^{i\theta} &\text{for }\theta\in [0,\pi]\\
e^{-i\theta} & \text{for }\theta\in [\pi,2\pi].
\end{cases}
\]
Clearly $\deg(v)=0$. We extend $v$ to the function on $\cl{D^2}$
\begin{equation}\label{defV}
V(z)=
\begin{cases}
z &\text{for }\mathrm{Im}z\ge 0\\
\bar z & \text{for }\mathrm{Im}z< 0
\end{cases}
\end{equation}
Let $\tilde v$ be the harmonic extension of $v$ to $D^2$. Then, since $V$ is not harmonic, we have
$$E(v)=\int_{D^2}|\nabla \tilde v|^2dxdy<\int_{D^2}|\nabla V|^2dxdy=2\pi.$$
Then, by conformal invariance, $E(v\circ\Pi_+^{-1})<2\pi$, and since $v\circ\Pi_+^{-1}\in \M{E}_{g_1,0}$, \eqref{eqgdeg0} is proven for $\lambda=1$.

By continuity, \eqref{eqgdeg0} continues to hold for $\lambda$ in a neighborhood of $1$, in particular in $[1,\lambda_0]$ for some $\lambda_0>1$.

For $\lambda \in (0,1)$ we still have \eqref{eqgdeg0}. Indeed, similar to \eqref{defV} we can construct
\begin{equation*} 
V_\lambda(z)=
\begin{cases}
\Phi_\lambda\circ \Psi (z) &\text{for }\mathrm{Im}z\ge 0\\
\Phi_\lambda\circ \Psi (\bar z) & \text{for }\mathrm{Im}z< 0,
\end{cases}
\end{equation*}
where $\Phi_\lambda:= \Phi\big(\frac{\cdot}{\lambda}\big)$, and $\Phi$ and $\Psi$ are as in \eqref{Phi} and \eqref{Psi}.

By symmetry and conformality we have, for $D^2_+=\{z\in D^2:\mathrm{Im}z> 0\}$,
$$\int_{D^2} |\nabla V_\lambda|^2 dxdy=2\int_{D^2_+} |\nabla V_\lambda|^2 dxdy=4 \mathrm{Area}(V_\lambda(D^2_+)) <2\pi,$$
since $V_\lambda$ sends the half disk $D^2_+$ into a strict subset of itself\footnote{To verify this, observe that $\Phi\circ\Psi=Id$, which of course sends $D^2_+$ onto itself, while $\Phi_\lambda\circ \Psi=\Phi\circ M_\lambda\circ\Psi$, where $M_\lambda (z)=\tfrac{z}{\lambda}$. On the other hand, $\Psi(D^2_+)=\R^2_+\setminus D^2_+$, and $M_\lambda$ for $\lambda\in (0,1)$, being a dilation, sends $\R^2_+\setminus D^2_+$ into a strict subset of itself, which is sent by $\Phi$ into a strict subset of $D^2_+$.} 
Setting $v_\lambda= V_\lambda|_{\sph^1}$, as before we have
$$E(v_\lambda\circ \Pi_+^{-1}) \le  \int_{D^2} |\nabla V_\lambda|^2 dxdy<2\pi,$$
and since $\deg(v_\lambda)=0$ and $v_\lambda\circ\Pi_+^{-1}=g_\lambda$ on $\R\setminus [-1,1]$, we obtain \eqref{eqgdeg0}.
\end{proof}

\section{Proof of Theorem \ref{thmconst}}

We will divide the proof in several steps.

\medskip

\noindent\textbf{Step 1: Definition and extension of the Hopf differential in $\mathbb{C}\setminus\{\pm 1\}$.}  We will use the Hopf differential, in a way similar to  \cite[Lemma 4.27]{milsir} and \cite[Sec. 4.2]{DLR3}. 

Let $\tilde u$ be the Poisson harmonic extension of $u$ to the upper half plane. Consider the Hopf differential
$$H(x,y)=\left(|\partial_x \tilde u|^2-|\partial _y \tilde u|^2\right)-2i\left(\partial _x \tilde u\cdot \partial_y \tilde u\right).$$
Notice that $E(u)=\|\nabla \tilde u\|_{L^2(\R^2_+)}^2<\infty$ implies $H\in L^1(\R^2_+,\mathbb{C})$.

Set $g(x,y)=\partial_x \tilde u\cdot \partial_y \tilde u$. Since $\tilde u (x,0)\equiv P$ for $x\in \R\setminus I$, we see that $g(x,0)=0$ for $|x|>1$. For $|x|<1$, first recall that $u$ is smooth in $I$ by \cite{DLR2}, so is $\tilde u$ in $\overline{R^2_+}\setminus \{(\pm 1,0)\}$. Then for $|x|<1$, $\partial_x \tilde u(x,0)=\partial_x u(x)\in T_{u(x)}\mathcal{N}$ and $\partial_y \tilde u (x,0)\perp T_{u(x)}\mathcal{N}$ by \eqref{eq12harmN}, hence $g(x,0)=0$ also for $|x|<1$. 


As $\tilde u$ is harmonic, it follows by a direct computation that $H$ is holomorphic on $\overline{\R^2_+}\setminus \{(\pm 1,0)\}$ 
Since $g=0$ on $\R\setminus \{(\pm 1,0)\}$, by the Schwarz  reflection principle we can extend $H$ to $\R^2\setminus\{(\pm 1,0)\}=\mathbb{C}\setminus \{\pm 1\}$ as $H(z)= \overline{H(\overline{z})}$ for $\Im z<0$. Of course we still have $H\in L^1(\mathbb{C}\setminus \{\pm 1\})$.

\medskip

\noindent\textbf{Step 2: Singularity removability and vanishing of $H$.} We start by claiming that $H$ has either removable singularities or simple poles at $\pm 1$. This follows at once from the following lemma.

\begin{lem}\label{lemsing} Let $h: D^2\setminus \{0\}$ be holomorphic and integrable. Then either
\begin{enumerate}
\item $h$ can be extended to a holomorphic function in $D^2$, or
\item $h$ has a simple pole at $0$, i.e. there is a holomorphic function $h_1: D^2\to \mathbb{C}$ with $h_1(0)\ne 0$ such that $h(z)=\frac{h_1(z)}{z}$ for $z\in D^2\setminus\{0\}$.
\end{enumerate}
\end{lem}

\begin{proof} We consider the Laurent series
$$h(z)=\sum_{k=-\infty}^\infty a_kz^k,\quad a_k:=\frac{1}{2\pi i}\int_{\partial B_r(0)}\frac{h(z)}{z^{k+1}}{dz}= \frac{1}{{2\pi r^k}}\int_0^{2\pi} h(r e^{i\theta})e^{-ik\theta}d\theta.$$
We estimate
$$2\pi |a_k| r^{k+1}\le r \int_0^{2\pi}| h(r e^{i\theta})|d\theta,$$
and integrate
$$\lim_{\ve\to 0}\int_{\ve}^1 2\pi |a_k| r^{k+1} dr \le \int_{D^2\setminus \{0\}}|h(z)||dz|^2<+\infty.$$
This implies at once $a_k=0$ for $k\le -2$. In the case $a_{-1}=0$, $h$ has a removable singularity, otherwise $h(z)=\frac{h_1(z)}{z}$ with
$$h_1(z)=\sum_{k=0}^\infty a_{k-1}z^k.$$
\end{proof}

Applying Lemma \ref{lemsing} to $H$ at $\pm 1$, we can write
$$H(z)=\frac{H_1(z)}{(z+1)(z-1)}$$
where $H_1:\mathbb{C}\to \mathbb{C}$ is holomorphic (it could be that $H_1(-1)=0$ and/or $H_1(1)=0$).
Using Cauchy's formula, for every $z_0\in \mathbb{C}$
\begin{equation}\label{H1'}
H_1'(z_0)=\frac{1}{2\pi i}\int_{\partial B_R(0)}\frac{H_1(z)}{(z-z_0)^2}dz=O\left(\int_{\partial B_R}\frac{|H_1(z)|}{|z|^2}|dz|^2\right),\quad \text{as }R\to\infty.
\end{equation}
Since $H\in L^1(\mathbb{C},\mathbb{C})$, we have
\begin{equation}\label{intH1}
\int_{\mathbb{C}\setminus D^2}\frac{|H_1(z)|}{|z|^2}|dz|^2<\infty,
\end{equation}
hence for a sequence of radii $R_k\to \infty$
$$\int_{\partial B_{R_k}(0)}\frac{|H_1(z)|}{|z|^2}|dz|\to 0,$$
which, together with \eqref{H1'} implies $H_1'(z_0)=0$. Then $H_1$ is constant, and  \eqref{intH1} implies $H_1\equiv 0$.
Then also $H\equiv 0$ in all of $\mathbb{C}$. In particular $\tilde u$ is conformal.

\medskip

\noindent\textbf{Step 3: Conformality implies constancy.} From conformality and $u\equiv P$ on $\R\setminus I$ we infer
\begin{equation}\label{tildeuconf}
0=\frac{\de \tilde u}{\de x}(x,0)=\frac{\de \tilde u}{\de y}(x,0),\quad \text{for }x\in \R\setminus [-1,1].
\end{equation} As $\tilde u$ is smooth on $\overline{\R^2_+}\setminus \{(\pm1,0)\}$,  we obtain 
$$\frac{\de^2 \tilde u}{\de y^2}(x,0)=-\frac{\de^2 \tilde u}{\de x^2}(x,0) =0\quad \text{for }x\in \R\setminus [-1,1].$$
It is then possible to extend $\tilde u(x,y)\equiv P$ for $y<0$, so that $\tilde u$ is harmonic on $\R^2\setminus([-1,1]\times \{0\})$ and constant on $\R^2_-$, hence $\tilde u\equiv P$ by analiticity. Therefore $u\equiv P$ in $\R$, and the theorem is proved. \hfill$\square$

\appendix

\section{Appendix}


\subsection{Fractional Laplacians and the spaces $\dot H^{\frac{1}{2},2}(\R,\mathbb{C})$, $H^{\frac{1}{2},2}(\sph^1,\mathbb{C})$}\label{a:fl}
For the proofs of the following statements, which are all classic, we refer to the Appendix of \cite{ManMar} and the references therein.

We start working on the real line.
Given $u\in \M{S}(\R,\mathbb{C})$ (the Schwartz space of rapidly decreasing complex-valued functions) and $s>0$, one can define $(-\Delta)^s u$ via Fourier transform as
\begin{equation}\label{Deltas}
(-\Delta)^s u =\M{F}^{-1}\(|\xi|^{2s}\hat u\),
\end{equation}
where
$$\M{F}(u)(\xi)=\hat u(\xi):=\frac{1}{\sqrt{2\pi}}\int_\R u(x)e^{-ix\xi}dx,$$
so that
$$u(x)=\M{F}^{-1}(\hat u)(x)=\frac{1}{\sqrt{2\pi}}\int_{\R} \hat u(\xi)e^{ix\xi}d\xi.$$
This definition can be extended by duality to every function in
$$L_s(\R,\mathbb{C}):=\left\{u\in L^1_{\loc}(\R,\mathbb{C}):\|u\|_{L_s}:=\int_{\R}\frac{|u(x)|}{1+|x|^{1+2s}}dx<\infty   \right\}, $$
in the sense of tempered distributions:
\begin{equation}\label{fraclapl2}
\langle (-\Delta)^s u,\varphi\rangle :=\int_{\R} u\overline{(-\Delta)^s \varphi} dx =\int_{\R}u\,\overline{\mathcal{F}^{-1}(|\xi|^{2s}\hat \varphi(\xi))}\,dx,\quad\text{for every }\varphi \in\mathcal{S}.
\end{equation}
Then we can define
\begin{equation}\label{defH12R}
\begin{split}
\dot H^{\frac{1}{2},2}(\R,\mathbb{C})&=\{u\in L_{\frac{1}{4}}(\R,\mathbb{C}): (-\Delta)^\frac14 u \in L^2(\R,\mathbb{C})\}\\
&=\{u\in L_{\frac{1}{4}}(\R,\mathbb{C}): |\xi|^\frac{1}{2} \hat u \in L^2(\R,\mathbb{C})\}.
\end{split}
\end{equation}
For $u\in\mathcal{S}(\R,\mathbb{C})$ one also has (see e.g. \cite[Section 3]{DPV})  that 
\begin{equation}\label{fraclapl}
(-\Delta)^s u(x)=K_s P.V.\int_{\R}\frac{u(x)-u(\xi)}{|x-\xi|^{1+2s}}d\xi:=K_s \lim_{\varepsilon\to 0}\int_{\R\setminus [-\ve,\ve]}\frac{u(x)-u(\xi)}{|x-\xi|^{1+2s}}d\xi,
\end{equation}
with $K_{\frac{1}{2}}=\frac{1}{\pi}$. In particular
\begin{equation}\label{FL}
(-\Delta)^\frac{1}{2}u(x)=\frac{1}{\pi} P.V.\int_{\R}\frac{u(x)-u(\xi)}{(x-\xi)^2}dy.
\end{equation}
Considering now the Poisson extension $\tilde u$ of $u$ given by
$$\tilde u(x,y)=\frac{1}{\pi}\int_\R\frac{yu(\xi)}{y^2+(x-\xi)^2}d\xi,\quad x\in \R, \,y>0,$$
we have (see e.g. \cite[Prop. A.1]{ManMar})
\begin{equation}\label{equivnorms}
\|(-\Delta)^\frac14 u \|_{L^2} =\|\nabla\tilde u\|_{L^2}.
\end{equation}
Indeed, for $u\in \M{S}(\R,\mathbb{C})$) we have
$$(-\Delta)^\frac12 u=- \frac{\de \tilde u}{\partial y}\bigg|_{y=0}, $$
hence with \eqref{Deltas} and integration by parts
\begin{equation}\label{FL2}
\|(-\Delta)^\frac14 u\|_{L^2}^2= \int_{\R} \bar u(-\Delta)^\frac12 u dx =-\int_{\R}\bar u \frac{\de \tilde u}{\de y}\bigg|_{y=0}dx =\int_{\R^2_+}|\nabla \tilde u|^2 dxdy.
\end{equation}
Then \eqref{equivnorms} follows by density.

We also have
\begin{equation}\label{equivnorms2}
\|(-\Delta)^\frac14 u \|_{L^2}^2= \frac{1}{2\pi}\int_\R\int_\R\frac{(u(x)-u(\xi))^2}{(x-\xi)^2}dxd\xi,
\end{equation}
which can be proven using \eqref{FL} to get
$$\int_\R \bar u(-\Delta)^\frac12 u dx=\frac{1}{\pi}\int_{\R}  \bar u(x)  P.V.\int_{\R}\frac{u(x)-u(\xi)}{(x-\xi)^2}d\xi dx$$
and exchanging the roles of $x$ and $\xi$ and summing to obtain
\begin{align*}
2\int_\R \bar u(-\Delta)^\frac12 u dx&=\frac{1}{\pi}\int_{\R}  \bar u(x)  P.V.\int_{\R}\frac{u(x)-u(\xi)}{(x-\xi)^2}d\xi dx  +\frac{1}{\pi}\int_{\R}  \bar u(y)  P.V.\int_{\R}\frac{u(\xi)-u(x)}{(x-\xi)^2}dx d\xi\\
&=\frac{1}{\pi}\int_{\R}\int_\R\frac{|u(x)-u(\xi)|^2}{(x-\xi)^2}dxd\xi.
\end{align*}
Then, using the first identity in \eqref{FL2}, \eqref{equivnorms2} follows.

\medskip

We can do similar computations for functions defined on $\sph^1$. We start by recalling that any map $u\in L^2(\sph^1,\sph^1)$ can be written in terms of its Fourier series
$$u(e^{i\theta})=\frac{1}{\sqrt{2\pi}}\sum_{k\in\mathbb{Z}}a_ke^{ik\theta},$$
where
\begin{equation}\label{deffoucoeff}
a_k=\frac{1}{\sqrt{2\pi}}\int_{\sph^1}u(e^{i\theta})e^{-ik\theta}d\theta.
\end{equation}
Defining the fractional Laplacian (in this case a half-derivative)
\begin{equation}\label{lapl14}
(-\Delta)^\frac14 u(e^{i\theta}):= \frac{1}{\sqrt{2\pi}}\sum_{k\in\mathbb{Z}}\sqrt{|k|}a_k e^{ik\theta},
\end{equation}
we can set in analogy with \eqref{defH12R}
\begin{equation}\label{defH12S1}
\begin{split}
H^{\frac{1}{2},2}(\sph^1,\mathbb{C})&=\left\{u\in L^2(\sph^1,\mathbb{C}): (-\Delta)^\frac14 u\in L^2(\sph^1,\mathbb{C})\right\}\\
&=\left\{u\in L^2(\sph^1,\mathbb{C}):\sum_{k\in\mathbb{Z}}|k||a_k|^2<\infty\right\},
\end{split}
\end{equation}
and
$$[u]_{H^{\frac{1}{2},2}(\sph^1)}^2=\sum_{k\in\mathbb{Z}}|k||a_k|^2.$$
Considering in polar coordinates $(r,\theta)$ the harmonic extension $\tilde u$ of $u$ to $D^2$
\begin{equation}\label{uPoisson}
\tilde u(re^{i\theta})=\frac{1}{\sqrt{2\pi}}\sum_{k\in\mathbb{Z}}a_kr^{|k|}e^{ik\theta},
\end{equation}
we can easily compute
\begin{equation}\label{equivnormsS1}
\|\nabla \tilde u\|_{L^2(D^2)}=[u]_{H^{\frac{1}{2},2}(\sph^1)}.
\end{equation}
On the other hand, we can also compute (see e.g. \cite[Proposition A.2]{DLMR}) 
\begin{equation}\label{FLS1}
(-\Delta)^\frac12 u(e^{i\theta})=\frac{1}{\pi}P.V. \int_{0}^{2\pi}\frac{u(e^{i\theta})-u(e^{it})}{|e^{i\theta}-e^{it}|^2}dt,
\end{equation}
and, similar to \eqref{equivnorms2} obtain
\begin{equation}\label{equivnorms3}
[u]_{H^{\frac{1}{2},2}(\sph^1)}^2=\int_{\sph^1}\bar u (-\Delta)^{\frac12}u\, d\theta= \frac{1}{\pi}\int_0^{2\pi}\int_0^{2\pi}  \frac{|u(e^{i\theta})-u(e^{it})|^2}{|e^{i\theta}-e^{it}|^2}dtd\theta.
\end{equation}
We also recall that the harmonic extension of $u$ can be expressed using the Poisson kernel:
\begin{equation}\label{uPoisson2}
\tilde u(z)=\frac{1}{2\pi}\int_{\sph^1}\frac{1-|z|^2}{|z- e^{i\theta}|}u(e^{i\theta})d\theta, \quad z\in D^2.
\end{equation}

Finally, we mention a characterisation of fractional Sobolev spaces in terms of traces of Sobolev maps on arbitrary bounded Lipschitz domains. For such a domain $\Omega\subset\R^2\simeq \mathbb{C}$ with $\Gamma:=\partial\Omega$, we define
$$[u]_{H^{\frac12,2}(\Gamma)}=\frac{1}{\pi}\int_\Gamma\int_\Gamma\frac{|u(z)-u(\zeta)|^2}{|z-\zeta|^2}d\sigma(z)d\sigma(\zeta),$$
where $d\sigma$ is the arc-length on $\Gamma$, and set
$$H^{\frac12,2}(\Gamma):=\left\{u\in L^2(\Gamma):[u]_{H^{\frac12,2}(\Gamma)}<\infty\right\}.$$ 
Then we have the following trace theorem, see e.g. \cite{Gagliardo,Prodi}:
\begin{prop}\label{trace}
For any bounded Lipschitz domain $\Omega$ with $\Gamma:=\partial\Omega$, the trace maps $u\mapsto u|_\Gamma$ is surjective from $H^{1,2}(\Omega)$ to $H^{\frac12,2}(\Gamma)$.
\end{prop}

\subsection{Density results}

\begin{prop}\label{l:approx}
The space of functions $u\in C^\infty(\R,\sph^1)$ which are constant outside a compact set is dense in $\dot H^{\frac{1}{2},2}(\R,\sph^1)$. Equivalently, the space of functions $u\in C^\infty(\sph^1,\sph^1) $  which are constant in a  neighborhood  (depending on $u$) of any fixed point $z_0\in \sph^1$ are dense in $H^{\frac{1}{2},2}(\sph^1,\sph^1)$.  
\end{prop}

In the proof of Proposition \ref{l:approx} we will borrow ideas from \cite{BN}. We first need two lemmas.

\begin{lem} \label{lem-Vitali} Let $(u_k)\subset H^{\frac12,2}(\sph^1,\mathbb{C})$ be such that $u_k\to u$ a.e. in $\sph^1$ for some $u\in H^{\frac12,2}(\sph^1,\sph^1)$. If $u_k\to u$ in $H^{\frac12,2}(\sph^1,\mathbb{C})$ and  $|u_k|\geq\delta$ a.e. in $\sph^1$ for some $\delta>0$ independent of $k$, then  $\frac{u_k}{|u_k|}\to u$ in $H^{\frac12,2}(\sph^1,\sph^1)$.  \end{lem}
\begin{proof} Setting  $\vp(z)=\frac{z}{|z|}-z$ we get that  $\vp$ is smooth in $\mathbb{C}^2\setminus\{0\}$, and $|\nabla \vp|\leq C(\delta)$ on $\mathbb{C}\setminus B_\delta$. The lemma follows if  we show that $\vp(u_k)\to \vp(u)\equiv0$ in  $H^{\frac12,2}(\sph^1,\mathbb{C})$.


We set $$g_k(z,\zeta):=\frac{|\vp( u_k(z))-\vp(u_k(\zeta))|^2}{|z-\zeta|^2},\quad z,\zeta\in\sph^1.$$
Then \begin{align*}
|g_ k(z,\zeta)|&\leq  C\frac{|u_k(z)- u_k(\zeta)|^2}{|z-\zeta|^2}\leq  C \frac{|(u_k(z)-u(z))-( u_k(\zeta)-u(\zeta))|^2}{|z-\zeta|^2}+ C\frac{|u(z)- u(\zeta)|^2}{|z-\zeta|^2}.\end{align*} This shows that the family $(g_k)$ is uniformly integrable in $\sph^1\times\sph^1$. In particular, as $g_k\to 0$ a.e. in $\sph^1\times\sph^1$, by Vitali's convergence theorem we get that $g_k\to0$ in $L^1 (\sph^1\times\sph^1)$, which gives $\vp(u_k)\to 0$ in  $H^{\frac12,2}(\sph^1,\mathbb{C})$. 
\end{proof}

\begin{lem}\label{l:approx0} Given  $u\in H^{\frac{1}{2},2}(\sph^1,\sph^1)$ and $\ve>0$ we set $$u_{\av,\ve}(e^{i\theta}):=\fint_{\theta-\ve}^{\theta+\ve} u(e^{it})\, dt.$$
Then $|u_{\av,\ve}|=1+o_\ve(1)$,  $u_{\av,\ve}\in C^{0,1}(\sph^1,\mathbb C)$, $u_{\av,\ve}\to u$ in  $H^{\frac12,2}(\sph^1,\mathbb{C})$ and $\frac{u_{\av,\ve}}{|u_{\av,\ve}|}\to u$ in  $H^{\frac12,2}(\sph^1,\sph^1)$. 
\end{lem}

\begin{proof} 
By the  Jensen's inequality
\[\begin{split}
(1-|u_{\av,\ve}(e^{i\theta})|)^2&\le\fint_{\theta-\ve}^{\theta+\ve}|u(e^{it})-u_{\av,\ve}(e^{i\theta})|^2dt\\
&=\fint_{\theta-\ve}^{\theta+\ve}\left|\fint_{\theta-\ve}^{\theta+\ve}u(e^{it})-u(e^{is})\,ds\right|^2 dt\\
&\le \fint_{\theta-\ve}^{\theta+\ve}\fint_{\theta-\ve}^{\theta+\ve}|u(e^{it})-u(e^{is})|^2\,ds dt\\
&\le \int_{\theta-\ve}^{\theta+\ve}\int_{\theta-\ve}^{\theta+\ve}\frac{|u(e^{it})-u(e^{is})|^2}{|e^{it}-e^{is}|^2}\,ds dt.\\
\end{split}\]
where in the last inequality we used that $|e^{it}-e^{is}|\le 2\ve$ for $s,t\in [\theta-\ve,\theta+\ve]$.
By \eqref{equivnorms3} and the absolute continuity of the Lebesgue integral, the last integral goes to $0$ as $\ve \to 0$ uniformly with respect to $\theta$, hence $|u_{\av,\ve}|\to 1$ uniformly as $\ve\to 0$. 

We now prove that $u_{\av,\ve} \to u$ in $H^{\frac12,2}(\sph^1,\sph^1)$. To this end we write \begin{align*} u_{\av,\ve}(e^{i\theta_1})- u_{\av,\ve}(e^{i\theta_2})=\fint_{-\ve}^\ve  \left(  u (e^{i(\theta_1+t)})-  u (e^{i(\theta_2+t)}) \right) dt.\end{align*} Then for any $A_1,A_2\subset (0,2\pi)$, \begin{align*} \int_{A_1}\int_{A_2} \frac{| u_{\av,\ve}(e^{i\theta_1})- u_{\av,\ve}(e^{i\theta_2})|^2}{|  e^{i\theta_1}-  e^{i\theta_2}|^2}d\theta_2 d\theta_1 & \leq \fint_{-\ve}^\ve \int_{A_1}\int_{A_2} \frac{|   u (e^{i(\theta_1+t)})-  u (e^{(i(\theta_2+t)})|^2}{|  e^{i(\theta_1+t)}-  e^{i(\theta_2+t)}|^2}d\theta_2 d\theta_1 dt\\ &= \fint_{-\ve}^\ve \int_{A_1+t}\int_{A_2+t} \frac{|   u (e^{i\theta_1 })-  u (e^{(i\theta_2 })|^2}{|  e^{i\theta_1 }-  e^{i\theta_2}|^2}d\theta_2 d\theta_1 dt.\end{align*} This shows that the family $ \frac{| u_{\av,\ve}(e^{i\theta_1})- u_{\av,\ve}(e^{i\theta_2})|^2}{|  e^{i\theta_1}-  e^{i\theta_2}|^2}$ is uniformly integrable in $(0,2\pi)\times(0,2\pi)$. Therefore, as $u_{\av,\ve}\to u$ a.e.,   by the Vitali's convergence theorem we get that $u_{\av,\ve}\to u$ in  $H^{\frac12,2}(\sph^1,\mathbb{C})$. 

Now the convergence $\frac{u_{\av,\ve}}{|u_{\av,\ve}|}\to u$ in  $H^{\frac12,2}(\sph^1,\sph^1)$ follows from Lemma \ref{lem-Vitali}.
\end{proof}

\noindent\emph{Proof of Proposition \ref{l:approx}} We first prove that the space of functions $u\in C^\infty(\sph^1,\sph^1) $  which are constant in a neighborhood of a fixed point, which we can assume to be $i$, is dense in $H^{\frac{1}{2},2}(\sph^1,\sph^1)$.  

Let $u\in H^{\frac{1}{2},2}(\sph^1,\sph^1)$. It follows from Lemma \ref{l:approx0} that it suffices to consider $u\in C^{0,1}(\sph^1,\sph^1)$. Let $\tilde u\in H^{1,2}(D^2)$ be the harmonic extension of $u$ in $D^2$. Writing the Poisson kernel as 
$$K(z,e^{i\theta}):=\frac{1-|z|^2}{2\pi |z-e^{i\theta}|^2}=\frac{1-r^2}{2\pi |re^{it}-e^{i\theta}|^2},\quad z=re^{it}\in D^2,$$
from \eqref{uPoisson2}, and using that $u\in C^{0,1}(\sph^1,\sph^1)$, we have
\[\begin{split}
\tilde u(re^{it})&=u(e^{it})+\int_{\sph^1}K(re^{it},e^{i\theta})\left(u(e^{i\theta})-u(e^{it})\right)d\theta\\
&=u(e^{it})+O\left(\int_{\partial D^2}K(re^{it},e^{i\theta})|e^{it}-e^{i\theta}|d\theta\right)\\
&=u(e^{it})+o(1)\quad \text{as }r\to 1.
\end{split}\]
In particular, $|\tilde u|\to 1$ in $D^2\setminus B_{1-\ve}$ uniformly as $\ve\to 0$, and $|\nabla \tilde u|\in L^\infty(D^2)$. Hence, one can find $U_k\in C^\infty(\bar D^2)$ such that $U_k\to U$ in $H^{1,2}(D^2)$, $|\nabla U_k|\leq C$, and $|U_k|\to 1$ as $k\to \infty$ uniformly on $\partial D^2$. 

For $\ve>0$ let $\eta_\ve\in C^\infty(\mathbb{C})$ be such that $\eta_\ve\equiv 0$ in $B_\ve (i)$, $\eta_\ve\equiv 1$ in $B^c_{2\ve}(i)$, $0\leq\eta_\ve\leq1$, and $|\nabla \eta_\ve|\leq\frac{4}{\ve}$. We now define
$$U_{k,\ve}\in C^\infty(\overline{D^2},\mathbb{C}), \quad U_{k,\ve}:=U_k\eta_\ve+U_k(i)(1-\eta_\ve).$$ Clearly each $U_{k,\ve}$ is constant in $B_\ve(i)$. Notice that
$$\nabla (U_{k,\ve}-U)=(U_k(i)-U_k)\nabla \eta_\ve+\nabla U_k(1-\eta_\ve),$$ which shows that $U_{k,\ve}\to U_k$ in $H^{1,2}(D^2)$ as $\ve\to0$. Moreover, as $|\nabla U_k|\leq C$, one easily obtains that $|U_{k,\ve}|=1+o_{k,\ve}(1)$ on $\partial D^2$. Thus, for some suitable choice of $\ve=\ve_k$, setting $u_k(z):=\frac{U_{k,\ve_k}(z)}{|U_{k,\ve_k}(z)|}$ on $\sph^1=\partial D^2$ we see that $u_k\to u$ in $H^{\frac{1}{2},2}(\sph^1,\sph^1)$, thanks to Lemma \ref{lem-Vitali}.

\medskip

The second part of the proposition follows from the equivalence of the Dirichlet energies
$$\int_{0}^{2\pi}\int_{0}^{2\pi}\frac{|f\circ\Pi_+(e^{it})-f\circ\Pi_+(e^{is})|^2}{|e^{it}-e^{is}|}dtds=\int_\R\int_\R\frac{|f(x)-f(y)|^2}{|x-y|^2}dxdy,$$ for any   $f\in \dot H^{\frac12,2}(\R,\sph^1)$,
which is turn follows from \eqref{equivnorms}, \eqref{equivnormsS1} and the fact that $\Pi_+$ can be extended to the conformal map $\Psi: \cl{D^2}\setminus\{i\}\to \R^2_+$ given in \eqref{Psi}.
\hfill$\square$

\medskip

Similar to the proof of Proposition \ref{l:approx}, one can approzimate functions in $H^{\frac12,2}(\sph^1,\sph^1)$ through  their extension (recalling that $H^{\frac{1}{2},2}(\sph^1,\mathbb{C})$ is the space of traces of functions in $H^{1,2}(D^2,\mathbb{C})$), as in the following lemma, which is a special case of \cite[Lemma 8 and Lemma A.13]{BN}.

\begin{prop}[\cite{BN}]\label{l:approx2} Given $U\in H^{1,2}(D^2,\mathbb{C})$ such that $U(\partial D^2)\subset\sph^1$, there exists a sequence $(U_j)\subset C^\infty(\overline{D^2},\mathbb{C})$ such that $U_j(\partial D^2)\subset \sph^1$ and $U_j\to U$ in $H^{1,2}(D^2,\sph^1)$.
\end{prop}

Via a conformal map from $D$ to $\R^2_+$, one obtains

\begin{prop}\label{l:approx3} Given $U\in H^{1,2}(\R^2_+,\mathbb{C})$ such that $U(\R\times\{0\})\subset\sph^1$, there exists a sequence $(U_j)\subset C^\infty(\overline{R^2_+},\mathbb{C})$ such that $U_j(\R\times\{0\})\subset \sph^1$ and $\nabla U_j\to \nabla U$ in $L^2(\R^2_+,\sph^1)$.
\end{prop}

\subsection{Two examples}

\begin{lem}\label{l:example1}
Let $\vp\in C^0_c(\R)\cap C^\infty(\R\setminus \{0\})$  be such that $\varphi(x)=0$ for $|x|\ge \frac12$ and for $x=0$,  and  
 $$\vp(x)=\frac{1}{\sqrt{\log\frac{1}{|x|}}},\quad \text{for } 0<|x|\leq\frac14.$$
 Then $\varphi\in H^{\frac12,2}(\R)$,
but $\vp\chi_{[0,\infty)}\not\in H^{\frac12,2}(\R)$.
\end{lem} 
 
\begin{proof}
To prove that $\vp\in H^{\frac12,2}(\R)$ we consider an extension $\Phi\in C^\infty_c(\bar\R^2_+\setminus\{0\} )$ with
$$\Phi(x,y)=\frac{1}{\sqrt{\log\frac{1}{|(x,y)|}}},\quad \text{for } 0<|(x,y)|\leq\frac14,\ y\ge 0.$$
Then $\nabla\Phi\in L^2(B_\frac14\cap \R^2_+)$, and $\Phi\in H^{1,2}(\R^2_+)$ follows by smoothness and compact support. Then $\vp\in H^{\frac12,2}(\R)$ by \eqref{equivnorms}, and $\|\nabla \tilde \varphi\|_{L^2}\le \|\nabla \Phi\|_{L^2}$, where $\tilde \varphi$ is the Poisson harmonic extension of $\varphi$.

To show that $\vp\chi_{[0,\infty)}\not\in H^{\frac12,2}(\R)$ we compute for any $0<\ve\le \frac{1}{4}$
\begin{equation}\label{stimavp}
\begin{split}
\int_{x=0}^\ve\int_{y=-\ve}^0\frac{(\vp(x)-\vp(y))^2}{(x-y)^2}dydx&=\int_0^\ve\left(\log\frac1x\right)^{-1} \int_{-\ve}^0\frac{dy}{(y-x)^2}dx \\ &=\int_0^\ve\left(\log\frac1x\right)^{-1}\left(\frac{1}{x}-\frac{1}{x+\ve}\right)dx=\infty,
\end{split}
\end{equation}
so that $[\vp\chi_{[0,\infty)}]_{H^{\frac12,2}}=+\infty$.
\end{proof}

\begin{prop}\label{example1} Let $\varphi\in H^{\frac12,2}(\R)$ be defined as in Lemma \ref{l:example1}. Then $u:= e^{i\varphi\circ \Pi_-}\in H^{\frac12,2}(\sph^1,\sph^1)\cap C^0(\sph^1,\sph^1)$, while the function
\[
w(e^{i\theta}):=
\begin{cases}
u(e^{i\theta}) &\text{for }\frac{\pi}{2}\le \theta\le \frac{3\pi}{2}\\
1 & \text{for} -\frac{\pi}{2}< \theta < \frac{\pi}{2}
\end{cases}
\]
belongs to $C^0(\sph^1,\sph^1)$ but not to $H^{\frac12,2}(\sph^1,\sph^1)$.
\end{prop}

\begin{proof} To prove that $u\in H^{\frac12,2}(\sph^1,\sph^1)$, it suffices to verify that
$$[u]_{H^{\frac{1}{2},2}(\sph^1)}^2= \frac{1}{\pi}\int_0^{2\pi}\int_0^{2\pi}  \frac{|u(e^{i\theta})-u(e^{it})|^2}{|e^{i\theta}-e^{it}|^2}dtd\theta<\infty,$$
which in turn, considering that $u$ is smooth away from $i\in \sph^1$, it reduces to verify that 
$$\int_{\frac{\pi}{2}-\ve}^{\frac{\pi}{2}+\ve}\int_{\frac{\pi}{2}-\ve}^{\frac{\pi}{2}+\ve}  \frac{|u(e^{i\theta})-u(e^{it})|^2}{|e^{i\theta}-e^{it}|^2}dtd\theta<+\infty,$$
for some $\ve>0$, which easily follows from
$$\int_{-1}^1\int_{-1}^1\frac{(\varphi(x)-\varphi(\xi))^2}{(x-\xi)^2}dxd\xi<+\infty$$
(Lemma \ref{l:example1}) and the fact that $\Pi_-$ is bi-Lipschitz from an arc of $\sph^1$ around $i$ and an interval in $\R$ around $0$.

Similarly, from \eqref{stimavp} it follows that $[w]_{H^{\frac12,2}(\sph^1,\sph^1)}=+\infty$, hence $w\not\in H^{\frac12,2}(\sph^1,\sph^1)$. The continuity of $u$ and $w$ follow easily by construction.
\end{proof}

\begin{lem}\label{l:example2} Define $\psi\in C^\infty(\R\setminus \{0\})$ such that
$$\psi(x)=\log\log\frac{1}{|x|}\quad \text{for }0<|x|<\frac14,$$
and $\psi(x)=0$ for $|x|\ge 1$. Then $\psi\in H^{\frac12,2}(\R)$.
\end{lem}
\begin{proof}
Similar to Lemma \ref{l:example1} it suffices to construct a function $\Psi\in C^\infty(\bar\R^2)$ with $\Psi(x,y)\equiv 0$ for $|(x,y)|\ge 1$ and
$$\Psi(x,y)=\log\log\frac{1}{|(x,y)|},\quad \text{for }|(x,y)|\le\frac14, \ y\ge 0.$$
Then $\nabla \Psi\in L^2(\R^2_+)$, hence $\psi=\Psi|_{\R\times \{0\}}\in H^{\frac12,2}(\R)$.
\end{proof}

With the same proof as for Proposition \ref{example1} we obtain:

\begin{prop}\label{example2} Let $\psi$ be as in Lemma \ref{l:example2}. Then $v:=e^{i\psi\circ\Pi_-}\in H^{\frac12,2}(\sph^1,\sph^1)$.
\end{prop}

\end{document}